\documentclass[11pt]{article}
\usepackage{graphicx}
\usepackage{amsfonts}
\usepackage{amsmath}
 \makeatletter
\def\p@equation{}
\makeatother

\newcommand{\R}{\mathbb R}

\newcommand{\Z}{\mathbb Z}

\newcommand{\QED}{\hspace*{\fill}\rule{2.5mm}{2.5mm}}
\newtheorem{theorem}{Theorem}
\newtheorem{lemma}[theorem]{Lemma}
\newtheorem{corollary}[theorem]{Corollary}
\newtheorem{conjecture}[theorem]{Conjecture}
\newtheorem{prop}[theorem]{Proposition}
\newenvironment{proof}{\noindent{\bf Proof\ }}{\QED\\}

\newcommand{\N}{\mathbb N}

\newtheorem{definition}{Definition}

\begin{document}

\begin{center}
{\LARGE {\bf Symmetric polynomials, $p$-norm inequalities, and certain functionals
 related to majorization.} }
 \medskip

{\Large {Ivo Kleme\v{s}} }

{\it \noindent Department of Mathematics and Statistics,
\noindent McGill University,
\noindent 805 Sherbrooke Street West,
\noindent Montr\'eal, Qu\'ebec,
\noindent H3A 2K6,
\noindent Canada.}
\noindent  Email:  klemes@math.mcgill.ca
\end{center}

\bigskip

\noindent {\it Abstract.}
We study the relation $ x \prec_L y$ on $[0,\infty)^n$ defined by
$ x \prec_L y \ \Leftrightarrow \
 \sum_{i=1}^n \psi(x_i) \leq \sum_{i=1}^n \psi(y_i)$
for all $\psi : [0,\infty)\to [0,\infty)$ of the form
$ \psi(s) = \int_0^s \ \varphi(t) \frac{dt}{t}$
where $\varphi$ is concave nondecreasing.
(We also briefly explain how this arises in the context of some $L^p$
inequalities between complex exponential sums conjectured by
Hardy and Littlewood, and why the more familiar relation
obtained by allowing any concave nondecreasing $\psi$
(a version of weak majorization) does not hold in that context.)
We attempt to characterize $x \prec_L y$ by
means of another relation, $x \prec_F y$, defined by
$ F_{k,r}(x) \le F_{k,r}(y) , \
\forall \ k,r \geq 1$, where  $F_{k,r}(x)$ is the coefficient of $t^k$ in
$\prod_{i=1}^n \left(1 + (x_it)^1/1! + \dots + (x_it)^r/r! \right)$.
We prove that\\
$x \prec_F y \Rightarrow x \prec_L y$. Regarding
the converse, we prove a necessary and sufficient condition
in terms of $\nabla g$ for a function $g$ to have the
order-preserving property $x \prec_L y \Rightarrow g(x) \leq g(y)$
``locally", and we verify that the condition holds for all $g=F_{k,r}$.
We then propose a general conjecture that the total positivity of
certain Jacobian matrices implies the ``path connectedness" of relations
such as $x \prec_L y$. If true, this conjecture would allow us to
remove the word ``locally'' and thus
complete the proof of $x \prec_L y \Rightarrow x \prec_F y$.

\vfill \noindent Research supported by NSERC Canada.

\noindent {\it A.M.S. Mathematics Subject Classification} :
52A40 (15A42, 15A48, 26B25, 42A05,  47A30, 60E15).


\noindent  \hfill Date: 9 April 2007.

\newpage

\noindent {\Large {\bf 1. Introduction.} }
\medskip

 For $x,y \in \R_{+}^n :=
[0,\infty)^n =\{ \ t \in \R^n \ | \ t_i \geq 0, \ i=1,\dots,n \ \}, $
$x$ is said to majorize $y$ ($x\succ y$) if  $\sum_{i=1}^n x_i = \sum_{i=1}^n y_i$
and $\sum_{i=1}^k \ x_i^* \ \geq \ \sum_{i=1}^k \ y_i^* $ for
$k=1, \dots , n-1,$ where $x_1^* \geq \dots \geq x_n^*$ denotes the
decreasing rearrangement of the entries $x_i$ of a vector $x.$
A real-valued function $F$ on $\R_{+}^n$ is called Schur-concave
if $x \succ y \Rightarrow F(x) \leq F(y) $, and Schur-convex
if $x \succ y \Rightarrow F(x) \geq F(y) $.
The relation (``pre-ordering") $x\succ y$
 is equivalent to the property that
$\sum_{i=1}^n x_i = \sum_{i=1}^n y_i$ and
$\sum_{i=1}^n \varphi(x_i) \geq \sum_{i=1}^n \varphi(y_i)$ for all
convex $\varphi : [0,\infty)\to \R$. Several other characterizations
are known \cite[p. 9-12]{MO}. In particular one can restrict $\varphi$ to
all nondecreasing convex $\varphi : [0,\infty) \to [0,\infty)$
with $\varphi(0) = 0,$ or
even to all $\varphi$ of the form $\varphi(s) = (s-\lambda)_+ $
for $\lambda > 0 .$ For $x,y \in \R_{+}^n$ with sums not necessarily
equal, we will denote by $x \prec^w y$ the relation defined by
$\sum_{i=k}^n \ x_i^* \ \leq \ \sum_{i=k}^n \ y_i^* $ for
$k=1, \dots , n$. (In \cite{MO} this is denoted by $y \prec^w x$.
We prefer $x \prec^w y$ because when $n=1$ this relation
reduces to the usual order $x \leq y$ on $\R_{+}$.) The relation
$x \prec^w y$ is equivalent to
\begin{equation}
\label{phi0}
\sum_{i=1}^n \varphi(x_i) \leq \sum_{i=1}^n \varphi(y_i)
\end{equation}
 for all
nondecreasing concave $\varphi : [0,\infty)\to [0,\infty)$ with
$\varphi(0) = 0,$ or just all $\varphi$ of the form $\varphi(s) = \min (s,\lambda)$
for $\lambda > 0 $.
Under the extra condition $\sum_{i=1}^n x_i = \sum_{i=1}^n y_i$ ,
it is clear that
$\ x \prec^w y$ is equivalent to $x \succ y$.
When the vectors $x,y $ are replaced by appropriate functions $f, g : [a,b]
\to [0,\infty), $
and sums by integrals, there are similar definitions and results.

Given any index set  $\Lambda$ and  a family of real-valued functions
$\{\Phi_\lambda\}_{\lambda \in \Lambda}$ defined on a subset $\mathcal{X} \subset \R^n$,
the relation defined for $x,y \in \mathcal{X}$ by the condition
$\Phi_\lambda(x) \leq \Phi_\lambda(y) \  \forall \lambda \in \Lambda$ will be called
the relation {\it induced by } $\{\Phi_\lambda\}_{\lambda \in \Lambda}$
on $\mathcal{X}$.
In this paper we study two specific examples of such relations
on $\R_{+}^n $, denoted
by $\prec_L$ and $\prec_F$ and defined below in
(\ref{L}) and (\ref{F}), which are both weaker than
$\prec^w$. In particular, $\prec_L$ is induced by
a certain subset of  the family of functions $\sum_{i=1}^n \varphi(x_i)$
 used above in (\ref{phi0}), and
$\prec_F$ is induced by
a certain subset (\ref{Fkr}) of  the family of all Schur-concave symmetric polynomials
with positive coefficients.
The relations
originate from a problem in harmonic
analysis  involving sharp $p$-norm relations of
the form $||x||_p \leq ||y||_p$
for various intervals of $p$ (the sharp Littlewood conjecture) \cite{K1}.
This problem will be summarized in \S2.2, but it
is not the main concern in this paper. Instead, we will be concerned
with understanding or characterizing the inequalities that define
$\prec_L$ and $\prec_F$ from various points of view.
A particular problem
is to determine the relation between the
two relations  $\prec_L$ and $\prec_F$. (In the application to
the sharp Littlewood conjecture, $\prec_F$ was introduced for the
specific purpose of understanding $\prec_L$ in a more computationally
accessible manner.)
We suspect that $\prec_L$ and $\prec_F$ are equivalent
and we give a proof of the implication $x \prec_F y  \Rightarrow x \prec_L y$
(Theorem 9). Regarding the converse, we give  ``local" results involving
gradients (Lemmas 12 and 13), analogous to the Schur-Ostrowski criterion concerning
$\prec^w$. We also discuss a general conjecture (Conjecture 16) which in particular
would allow one to deduce the full converse
``$x \prec_L y  \Rightarrow x \prec_F y$"
from our local results.

 Theorems 1, 7, 9 and Proposition 3 were previously
 reported in \cite{K2}.
\bigskip

\noindent {\Large {\bf 2. Results and background.} }
\medskip

{\bf 2.1.}
 Define the relation $x \prec_L y$ on $\R_{+}^n$ by
\begin{equation}
\label{L}
x \prec_L y \ \Leftrightarrow \
 \sum_{i=1}^n \psi(x_i) \leq \sum_{i=1}^n \psi(y_i)
 \end{equation}
for all $\psi : [0,\infty)\to [0,\infty)$ of the form
\begin{equation}
\label{psi}
 \psi(s) = \int_0^s \ \varphi(t) \frac{dt}{t}  \ , \ \ \ \ \ \ (s \geq 0),
 \end{equation}
for some nondecreasing concave $\varphi : [0,\infty)\to [0,\infty)$.
This definition does not change if we restrict
$\varphi$ to just those $\varphi$ of the form $\varphi(s)=\varphi_\lambda(s) :=
\min(s,\lambda), \ \lambda > 0$ (by the same argument that justifies this
restriction  in (\ref{phi0})). Integrating shows that this choice of $\varphi$
gives $\psi(s) = \psi_\lambda(s)
:= \min(s,\lambda) + \lambda \log_+(s/\lambda).$
 Thus $x \prec_L y \Leftrightarrow \Psi_\lambda(x) \leq
\Psi_\lambda(y) \ \forall \lambda > 0$, where
\begin{equation}
\label{psib}
 \Psi_\lambda(x) :=
 \sum_{i=1}^n \  \psi_\lambda(x_i)
 =  \sum_{i=1}^n  \  \min(x_i,\lambda) + \lambda \log_+(x_i/\lambda)
 \ , \ \ \ \ \ \ (\lambda > 0, \ x \in \R_{+}^n).
 \end{equation}
It can be checked that any $\psi$ of the form (\ref{psi}) is concave,
whence $x \prec^w y \Rightarrow x \prec_L y.$ (The converse
does not hold when $n\geq 2$ ; see \S7.)

If $\sum_{i=1}^n x_i = \sum_{i=1}^n y_i$ , then $x \prec_L y$ is equivalent
to $\sum_{i=1}^n \psi(x_i) \geq \sum_{i=1}^n \psi(y_i)$ for all
$\psi : [0,\infty)\to [0,\infty)$ of the form
$\psi(s) = \int_0^s \ \varphi(t) \frac{dt}{t}$ for some
 nondecreasing {\it convex } $\varphi : [0,\infty)\to [0,\infty)$
 (as can be seen using the
 identity $(t-\lambda)_+ = t - \min(t, \lambda)$).
In this case, we can also view the relation $x \prec_L y$ as an example
of ``tensor-product-assisted majorization"  (or the ``trumping relation")
\cite{JP}, \cite{N}, \cite{DK}, except that
here the ``catalyst" is a function on a measure space
instead of a vector in some $\R_{+}^d$. We explain this briefly in \S8.

We want to study $\prec_L$ because it implies the following
$l^p$ inequalities in $\R_{+}^n$:
\begin{equation}
\label{p1}
x \prec_L y \ \ \Rightarrow \ \  ||x||_p \leq ||y||_p \ ,
 \ \ \ 0 \leq p \leq 1 \ (p \in \R).
 \end{equation}
\begin{equation}
\label{p2}
 x \prec_L y \ , \  \sum_{i=1}^n x_i  = \sum_{i=1}^n y_i \  \ \Rightarrow \
 \ ||x||_p \geq ||y||_p \ , \ \ \ 1 \leq p \leq \infty \ (p \in \R).
\end{equation}
These follow immediately by the above remarks and the fact
that  $\varphi(s) = s^p$ is
concave in $s$  when
$0 < p \leq 1$, convex when $1 \leq p < \infty$, and
 is an eigenfunction
  for the operator $L[\varphi](s):= \int_0^s \ \varphi(t) \frac{dt}{t} \ $
for all  $0 < p < \infty$ (with eigenvalue $1/p >0$).
(If $p<0$ one can show that the implication (\ref{p1}) fails for some $n, x,y$, even if
$\sum x_i = \sum y_i$. This contrasts with the fact that if $x\succ y$
then $||x||_p \leq ||y||_p$ for all $p<0$ (since $x^p$ is convex).)
%%%%%%%%%%%%%%%%%

%%%%%%%%%%%%%%%
We now summarize a problem from \cite{K1} in order to further
explain our interest in $\prec_L$ and to motivate the relation
$\prec_F$ to be defined in \S2.3.

{\bf 2.2.} For any fixed $N \geq 1,$
Hardy, Littlewood, and Gabriel \cite{HL},\cite[Ch. 2]{P} proved the sharp $L^{2p}$ inequalities
\begin{equation}
\label{p}
||D_N||_{2p} \geq ||f||_{2p} \ , \ \ \ \ p=1,2,3, \dots ,
\end{equation}
where $f(z) = c_0 + c_1 z^{n_1} + \dots + c_{N-1} z^{n_{N-1}} $
 ( $ \ 0 < n_1 < \dots < n_{N-1} $ ) is any $N$-term trigonometric
polynomial on the circle $\{ z= e^{i\theta} \}$
whose coefficients are complex and satisfy $|c_i| = 1$,
and where $D_N$ is the special case $D_N(z) := 1+z+ \dots + z^{N-1}.$
This led to the conjectures that
\begin{equation}
\label{pa}
{\bf (a)} \  ||D_N||_{2p} \geq ||f||_{2p}\ , \ \forall p \in [1,\infty] \ (p \in \R),
\ \ \ {\rm and} \ \ {\bf (b)} \
||D_N||_{2p} \leq ||f||_{2p}\ , \ \forall p \in [0,1] \ (p \in \R).
\end{equation}
 (Note that $||D_N||_2 = ||f||_2 = \sqrt N$.) The case
$p=\frac{1}{2}$ in (\ref{pa})(b) came to be called the sharp Littlewood conjecture.
(The ``non-sharp" Littlewood conjecture is the case $p=\frac{1}{2}$
but with some absolute positive constant multiplier, i.e. $C||D_N||_{1} \leq ||f||_{1}$.
It was proved by Konyagin \cite{Ko} and independently
by McGehee et al. \cite{MPS}.)
The inequalities in (\ref{pa})((a) and (b)) would certainly follow from the majorization relation $ |D_N|^2
\succ |f|^2$ for these functions on the circle, but
unfortunately $ |D_N|^2 \succ |f|^2$ does not actually hold in all
cases. This was in effect remarked by
Domar and Pichorides \cite[Ch. 3, p. 42]{P},
who noted that counter-examples to $ |D_N|^2 \succ
|f|^2$ can be found simply by looking for an $f$ having a double root on
the unit circle. $D_N$ has only simple roots. Thus $f(z) =
(1-z)(1-z^2)= 1-z-z^2+z^3$ provides a counter-example for $N=4.$
Similar multiple-root examples also show that (\ref{pa})(b) cannot
be extended to any $p<0$, where $||f||_{2p}$ for negative $p$ would still be defined
by the formula $||f||_{2p} := (\int_0^{2\pi} |f|^{2p} d\theta/2\pi )^{1/(2p)}$.

The motivation for the present paper is the possibility
that, in light of the Domar-Pichorides remark, there may exist some other, weaker,
distribution-function relation between $ |D_N|^2$ and $ |f|^2$
which still implies (\ref{pa}),
but which is somehow more basic than just (\ref{pa}).
 It was conjectured in \cite{K1} that this relation
is $ |D_N|^2 \prec_L |f|^2$, that is, $ \int_0^{2\pi}
\psi(|D_N(e^{i\theta})|^2) d\theta/2\pi \leq  \int_0^{2\pi}
\psi(|f(e^{i\theta})|^2) d\theta/2\pi$ for all $\psi$ of the form
(\ref{psi}). This does of course imply all of the conjectured
$L^{2p}$ inequalities in (\ref{pa}) - by the function versions of
remarks (\ref{p1}) and (\ref{p2}). Let us also note that it is
natural to extend the conjecture (\ref{pa})(b), as well
as the conjecture $ |D_N|^2 \prec_L |f|^2$,
to $f$ with coefficients satisfying $|c_i| \geq 1$.
In this case one does not have $ \int |D_N|^2 = \int |f|^2$.
Because of this, we have not included the sum condition
$\sum x_i = \sum y_i$ when defining $\prec_L$.
We now discuss some of the motivation and evidence for the conjecture
 $ |D_N|^2 \prec_L |f|^2$.

One reason for $\prec_L$ stems from a weaker relation
proved in \cite[Theorem 1.2]{K1},
to the effect that
\begin{equation}
\label{log}
 \int_0^{2\pi} \log(1+ t|D_N(e^{i\theta})|^2 )d\theta/2\pi \
 \leq \  \int_0^{2\pi} \log(1+t|f(e^{i\theta})|^2 )d\theta/2\pi \
 \ \ \ \ \ \ \ \ \forall t \geq 0,
 \end{equation}
 whenever $f$ has the special form $f(z) = \sum_{i=0}^{N-1} \pm z^i$.
This relation was originally defined by means of
a certain ``extremum functional" (\ref{Iinf}). When that
functional is generalized, it
leads in a natural way to consideration of the stronger relation $\prec_L$.
These functionals will be discussed in \S5.
(The question as to whether the relation $\prec_L$ holds
between $|D_N|^2$ and $|f|^2$, even for these special
$f$'s, is still open and will not be pursued.)

A second reason concerns the above $N=4$
counter-examples to the majorization relation, of the type $f(z) = (1-z)(1-z^2)$.
A proof of the conjectured $L^{2p}$ inequalities (\ref{pa}) for
these examples had been found a number of years before
the conjecture $ |D_N|^2 \prec_L |f|^2$ was proposed.
When this proof was re-examined in view of the
conjecture, it was found that the proof
already yields the conjecture for these examples \cite[\S3]{K1}.

Thirdly, it was further conjectured in \cite {K1} that the
relation $ |D_N|^2 \prec_L |f|^2$ actually results
as a limiting case of the discrete version $ x \prec_L y$,
where $x =\overrightarrow{\lambda}(A)
:= (\lambda_1(A), \dots, \lambda_n(A)) $ and $y=\overrightarrow{\lambda}(B)
= (\lambda_1(B), \dots, \lambda_n(B))$
 are the eigenvalues of the $n \times n$
Toeplitz matrices $A$ and $B$ generated by $ |D_N|^2$ and $ |f|^2$
respectively. As $n\to \infty$, the relation  $ |D_N|^2 \prec_L |f|^2$
 would follow by a limit theorem
of Szeg\"o. Then in
\cite{K2} it was noted that
the discrete relation $x \prec_L y$ arises in a way that seems
very natural for the matrix context, in the sense
that it is equivalent to the following (see Theorem 7 (b)):
\begin{equation}
\label{sup}
 \sup_{||z||_1 \leq 1} \ \prod_{i=1}^n (z_i+t x_i) \ \leq \
\sup_{||z||_1 \leq 1} \ \prod_{i=1}^n (z_i+ty_i) \ , \ \ \  \forall t > 0. \
\end{equation}
If $x =\overrightarrow{\lambda}(A)$ and $y =\overrightarrow{\lambda}(B)$
for some $n \times n$ Hermitian $A, B \geq 0$, the relation $x \prec_L y$
can be further restated as just
$$
\sup_{tr(Z) \leq 1} \det (Z+tA) \ \leq \ \sup_{tr(Z) \leq 1} \det (Z+tB)
\ , \ \ \  \  \forall t > 0,
$$
where $Z$ runs over all $n \times n$ Hermitian $Z \geq 0$ (see (\ref{supdet})).
Moreover, when $n=3$ and
$\sum x_i = \sum y_i$, the relation
$ x \prec_L y$ is equivalent to the two simple conditions
$\prod x_i \leq \prod y_i$ and
$\max(x_i) \geq \max(y_i)$ (as mentioned in
\cite[\S3.3]{K1}; see Corollary 18 below).
Thus, one could argue that this ``simplicity" of $\prec_L$,
which seems to be emerging in the discrete case,
is itself a good reason to consider the conjecture
$ |D_N|^2 \prec_L |f|^2$ and the discrete approach to it.

On the other hand, the simplicity of (\ref{sup})
is perhaps only an appearance of simplicity. To make
 (\ref{sup}) more ``computable", we consider instead of (\ref{sup})
 a family of symmetric polynomials $\{ F_{k,r}(x) \}$ that extends
the family of elementary symmetric polynomials. This idea
is an extension of the strategy in \cite{K1} of using elementary symmetric
polynomials to prove the discrete version of (\ref{log}).
We now introduce the family $\{ F_{k,r}(x) \}$ and the new relation $x \prec_F y$
 based on it.

%%%%%%%%%%%%%%%%%%%%%%%%%%%
{\bf 2.3.}
Consider the relation $\prec_E$ on $ \R_{+}^n$ defined by
$x \prec_E y  \ \Leftrightarrow \ E_k(x) \le E_k(y) , \ k=1,\dots,n,$
where
$$E_k(x) = \sum_{1\leq i_1 < \dots < i_k \leq n} x_{i_1}\dots x_{i_k}  $$
 is the elementary symmetric polynomial of degree $k$, i.e. the coefficient
 of $t^k$ in the generating function
 $f_1(x,t):= \prod_{i=1}^n \left(1 + x_it \right)$.
The relation $x \prec_E y$ clearly implies that
$\sum_i \log(1+x_it) \leq \sum_i \log(1+y_it), \ \forall t\geq 0.$
This implies some of the $l^p$ inequalities seen in (\ref{p1})
and (\ref{p2}), but not all of them. It implies
$ ||x||_p \ \leq  ||y||_p  \ , \ 0\leq p \leq 1, $
and if  $\sum x_i = \sum y_i, $ then also
$ ||x||_p \ \geq  ||y||_p  \ , \ 1\leq p \leq 2. $
 (To prove this, one may put $r=1$ in the proof of Theorem 1 below.)
 The restriction $p \leq 2$ in the latter implication is sharp when $n \geq 3$, as is the
 restriction $p \geq 0$ in the former \cite[\S2]{K1}.
These remarks were used in \cite{K1} in connection with the
result (\ref{log}), and they have been used earlier in other applications
 as well \cite[p. 211-212, Lemma 11.1, Ch. 4]{GK}, \cite[Theorem 4]{MS}. We now define
the larger family of symmetric polynomials $\{ F_{k,r}(x) \}$ as follows.
\begin{definition}
For $k,r \geq 1,$ let $F_{k,r}$ denote the polynomial of degree
$k$ in $n$ variables given by
\begin{equation}
\label{Fkr}
F_{k,r}(x_1,\dots,x_n)\  = \ \ \sum_{\sum k_i = k, \ \max k_i \leq r}
\ \ \prod_{i=1}^n \frac{x_i^{k_i}}{k_i!}\ ,
\end{equation}
where it is understood that the $k_i$ range over the nonnegative integers.
Also, put $F_{0,r}=1.$
\end{definition}
The $F_{k,r}(x)$ are given by a generating function
$f_r(x,t)$ which generalizes
the above generating function $f_1(x,t)$
 for the $E_k = F_{k,1}$. Let $P_r$ be the $r$th degree Taylor polynomial
of $\exp$, that is $P_r(s) = 1 + s + \dots + \frac{s^r}{r!}$.
Then $F_{k,r}(x)$ is the coefficient of $t^k$ in the polynomial
\begin{equation}
\label{gen1}
f_r(x,t):=
\prod_{i=1}^n \ P_r(x_it) =
\prod_{i=1}^n \left(1 + x_it + \dots + \frac{x_i^r}{r!}t^r
\right).
\end{equation}
 We have
$F_{k,1} = E_k \ ,$  $F_{k,r} = (E_1)^k/k! \ $
for $k \leq r$, $ F_{nr,r} = (E_n)^r/(r!)^n$, and $F_{k,r}=0$  for $k > nr$.
Here $n$ is the number of variables; we may also view $x$
as an infinite sequence of the form $x= (x_1,x_2, \dots , x_n, 0,0, \dots)$.
\begin{definition}
The relation $\prec_F$ is defined on $ \R_{+}^n$ by
\begin{equation}
\label{F}
x \prec_F y  \ \Leftrightarrow \ F_{k,r}(x) \le F_{k,r}(y) , \
\forall \ k,r \geq 1.
\end{equation}
\end{definition}
\noindent
Clearly $x \prec_F y \Rightarrow x \prec_E y \ , $ since $\{E_k\} \subset \{F_{k,r}\} \ .$
The $F_{k,r}$ belong to an even larger family of symmetric polynomials
$\{H_S\}$ that were all proved to be Schur-concave on $ \R_{+}^n$
by Proschan and Sethuraman. Thus $x \succ y
\Rightarrow x \prec_F y$. (Consequently, $x \prec^w y
\Rightarrow x \prec_F y$, by general results \cite[Theorem 3.A.8]{MO}.)
We discuss $\{H_S\}$ in \S4, where
we also note that a certain other subfamily
$\{G_{k,r} \} \subset \{H_S\}$  characterizes majorization  on
$ \R_{+}^n$ ,
in the sense that $x \succ y \Leftrightarrow
G(x) \le G(y) , \ \forall \ G \in \{G_{k,r} \}$, when $\sum x_i = \sum y_i$
(Proposition 3).
The family $\{G_{k,r} \}$ and the latter proposition arise naturally
from the trivial remark that for
any integer  $r \geq 1$ and $x \in \R_{+}^n$ we have
$$ \sum_{i=1}^r \ x_i^* \ = \
\sup_{1\leq i_1 < \dots < i_r \leq n} \
(x_{i_1} + \dots + x_{i_r})
\ =  \ \lim_{k\to \infty \ (k \in \N)} \
\left( \sum_{1\leq i_1 < \dots < i_r \leq n}
\ (x_{i_1} + \dots + x_{i_r})^k \ \right)^{\frac{1}{k}}.
$$

A similar device for computing the supremum  in  (\ref{sup})
leads us to the family $\{F_{k,r} \}$ (see \S2.4).
But, in contrast to the family $ \{G_{k,r} \}$,
the relation induced by the family $\{F_{k,r} \}$ is strictly weaker: For $n \geq 3$, $x \prec_F y$
does not imply $x \succ y$ when
$\sum x_i = \sum y_i$ (by Corollary 18 applied to, for example,
$x=(15,2,2),\ y=(9,9,1)$). However, in Theorem 1 we will prove that $x \prec_F y$
is strong enough to imply all
of the $l^p$ inequalities in (\ref{p1})
and (\ref{p2}).
Hence, if $\sum x_i = \sum y_i$, then $\prec_F$ overcomes the
restriction $p \leq 2$ mentioned
above in connection with $\prec_E$. Furthermore,
Theorem 1 shows that the sharp restriction ``$p \leq 2$"
can be progressively weakened to ``$p \leq (r+1)$" for any integer $r=1,2,3, \dots$
by using progressively larger subsets of the family $\{F_{k,r} \}$.
Then in Theorem 9 we prove that
the full family $\{F_{k,r} \}$ gives even
more than all of the $p$-norm inequalities in (\ref{p1})
and (\ref{p2}):
we prove the stronger implication $x \prec_F y \Rightarrow x \prec_L y$.
 As mentioned in \S1, we suspect that the converse holds as well,
and we give a fairly compelling local result on this (Corollary 14).
Let us outline some of the details.

{\bf 2.4.}
The proof of $x \prec_F y \Rightarrow x \prec_L y$ (Theorem 9)
has three main ingredients.
The first is to
prove that $x \prec_L y$ is equivalent to the above supremum relation (\ref{sup})
(Theorem 7 (b)).
The second is to observe that
$$
\sup_{||z||_1 \leq 1} \ \prod_{i=1}^n (z_i+tx_i) \
= \ \lim_{r \to \infty \ (r\in \N)}
\left(
\int_{\sum z_i \leq 1 \ (z_i \geq 0)} \ \bigg( \prod_{i=1}^n ( z_i + tx_i ) \bigg)^r \
dz_1 \dots dz_n
\right)^{\frac{1}{r} } \ .
$$
The third is the following alternative ``generating function"
for the $F_{k,r}(x)$  (Lemma 8):
$$
\int_{\sum z_i \leq 1 \ (z_i \geq 0)} \ \bigg( \prod_{i=1}^n ( z_i + tx_i ) \bigg)^r \
dz_1 \dots dz_n
\
= \  \sum_{k=0}^{nr}C_{n,k,r} F_{k,r}(x)t^k \
$$
where the $C_{n,k,r}$ are some positive constants
($C_{n,k,r} =\frac{(r!)^n }{(nr+n-k)!}$). It was this generating function,
rather than (\ref{gen1}),
which originally led the author to consider the family $\{F_{k,r}\}$.

Regarding the conjectured converse direction, $x \prec_L y \Rightarrow x \prec_F y$,
our ``local" result (Corollary 14) is analogous to the
 use of gradients and Jacobians in the paper of
Marshall et al. \cite[Corollaries 7,8]{MWW},
except that we have not yet
succeeded in ``globalizing" our conclusions.
To motivate the statement of the local result, let us state
the Schur-Ostrowski criterion \cite[Theorem 3.A.7]{MO},
\cite[Example 1]{MWW} in the following way:
A $C^1$ symmetric function $\Phi :\R_+^n \to \R$
has the order-preserving property $x \prec^w y \Rightarrow
\Phi(x) \leq \Phi(y) $ on $\R_+^n$ if (and only if)
the matrix
\begin{equation}
\label{S-O}
\left[
\begin{array}{cc}
 1 & 1 \\
  \frac{\partial \Phi}{\partial x_i}(x)
 & \frac{\partial \Phi}{\partial x_j}(x)
\end{array}
\right]
\end{equation}
is totally positive whenever $ i \leq j$ and
$x \in \mathcal{D}_n := \{ x \in \R^n \ | \ x_1 \geq \dots \geq x_n \geq 0\}$.
(A matrix is said to be totally positive if all of its subdeterminants are
nonnegative).  The criterion is ``local" in the sense that
only derivatives are involved, yet the result is global in the
sense that $x$ and $y$ can be arbitrarily far apart.
The  criterion is equivalent to saying that at any point
 $x \in \mathcal{D}_n$,
$\nabla \Phi(x)$ belongs to the positive cone generated by
the $\nabla (x_k + x_{k+1} + \dots + x_n), \ k = 1, \dots,  n$.
(The positive cone is the closure of the set of finite
linear combinations using positive coefficients.) Recall from \S1 that
the functions $\phi_k(x):=(x_k + x_{k+1} + \dots + x_n)$
induce the relation $x \prec^w y$
(when $x,y \in \mathcal{D}_n$) via the conditions $\phi_k(x) \leq \phi_k(y)$.
We prove an analogous gradient property regarding
the functions $\Psi_\lambda(x)$ (\ref{psib}) that induce the relation
$x \prec_L y$ (Lemma 12). The property states that for a fixed
$x \in \mathcal{D}_n$ with say $x_n > 0$,
a vector $B \in \R^n$ is in the positive cone generated
by the $\{\nabla \Psi_\lambda(x)\}_{\lambda > 0}$ if and only if
the matrix
$$
\left[
\begin{array}{ccc}
1 & 1 & 1 \\
B_i & B_j & B_k \\
  \frac{1}{x_i} & \frac{1}{x_j} & \frac{1}{x_k}
\end{array}
\right]
$$
is totally positive
whenever $ i \leq j \leq k$.
We then verify that this holds when $B = \nabla F_{k,r}(x)$
(Lemma 13).
This local result would give
the full global result ``$x \prec_L y \Rightarrow x \prec_F y$" if
it was known that $\prec_L$ has a certain path-connectedness property
(Theorem 15), which we suspect it does have.
Conjectures on the latter
can be found near the end of \S6, where we speculate that such
a connectedness property may be just a general consequence
of the total positivity of certain gradient matrices, i.e. Jacobians
(Conjecture 16). The methods of \cite{Ger} may be applicable to this
conjecture. It is however also possible
that the global result $x \prec_L y \Rightarrow x \prec_F y$ follows
from the local one (Corollary 14) by some other general argument
(or that it is false!).

{\bf 2.5.} Be that as it may, the truth of the implication
 $x \prec_L y \Rightarrow x \prec_F y$
is not really of concern in the context of the conjecture
that $ \overrightarrow{\lambda}(A) \prec_L \overrightarrow{\lambda}(B)$
 discussed in \S2.2, although it
would still be interesting to settle the issue.
Instead, it is the result $x \prec_F y \Rightarrow x \prec_L y$ (Theorem 9) that
should be the more useful one, if one  first
proves that $ \overrightarrow{\lambda}(A) \prec_F \overrightarrow{\lambda}(B)$
for the relevant matrices $A,B$. The problem
$ \overrightarrow{\lambda}(A) \prec_F \overrightarrow{\lambda}(B)$
seems to be more tractable than the problem
$ \overrightarrow{\lambda}(A) \prec_L \overrightarrow{\lambda}(B)$,
in that $\prec_F$ only depends on computing
polynomials in the entries of $A$ and $B$.
(In that sense, it is also a return to the spirit of the original
results (\ref{p}), which involve polynomials in the coefficients
of Fourier series.)
For example, the author has proved that
$F_{4,2}(\overrightarrow{\lambda}(A)) \leq F_{4,2}(\overrightarrow{\lambda}(B))$
for some of the $4 \times 4$ Toeplitz matrices $A,B$ related to the
conjectures in \S2.2.
 The proof begins with the identity $F_{4,2} = (E_2^2 - 2E_4)/4$,
 and, after some Binet-Cauchy expansions,
 reduces to determinant inequalities
  similar to the Alexandrov
 inequalities used in \cite{K3} (where it was proved that $F_{n,1}(\overrightarrow{\lambda}(A))
 \leq F_{n,1}(\overrightarrow{\lambda}(B))$,
 i.e. $\det A \leq \det B$, for some cases of $A$ and $B$).
 The details will not be discussed in this paper.
 \bigskip

\noindent {\Large {\bf 3. The $l^p$ inequalities.} }
\begin{theorem}
Let $x,y \in \R_{+}^n$ and fix an integer $r \geq 1. $
If $ F_{k,r}(x) \leq  F_{k,r}(y) $
for all integers  $k$
in the interval $r \leq k \leq nr,$ then
 $$ ||x||_p \leq ||y||_p \ , \ \ \ 0 \leq p \leq 1, \ (p \in \R),$$
and, if in addition $\sum_i x_i  = \sum_i y_i \ ,$ then
$$ ||x||_p \geq ||y||_p \ , \ \ \ 1 \leq p \leq r+1, \ (p \in \R).$$
\end{theorem}

\begin{proof} Fix the integer $r\geq1.$
Observe that $\log(1 + s + \dots + \frac{s^r}{r!})$ is $O(s)$ when
$s\to 0^{+}$ and $O(\log s)$ when $s \to +\infty.$
Thus, the  integrals
$$I_r(p)  := \int_0^\infty \log(1 + s + \dots + \frac{s^r}{r!})
\ s^{-p} \  \frac{ds}{s}$$
are finite (and positive) for all real $p$ in the interval $0<p<1.$ Replacing $s$
by $at$ for any positive $a$ gives the identity
\begin{equation}
\label{id1}
\frac{1}{I_r(p)}\int_0^\infty \log(1 + at + \dots + \frac{(at)^r}{r!})
\ t^{-p} \  \frac{dt}{t} \ = \ a^p \ \ \  \ (a \geq 0,\ 0<p<1).
\end{equation}
 If now
$x,y \in \R_{+}^n$ and
 $ F_{k,r}(x) \leq  F_{k,r}(y) $
for all integers  $k$
in the interval $r \leq k \leq nr,$ then  $ F_{k,r}(x) \leq  F_{k,r}(y) $
also holds for all $k$ in the interval $0\leq k <r$ (since
 $F_{k,r} = (E_1)^k/k! \ $ for $k \leq r$), hence by (\ref{gen1}),
$$1\leq f_r(x,t) \leq f_r(y,t) \ , \forall \ t \geq 0 .$$
Taking logarithms of the $f_r$ and integrating with respect to $t^{-p} \  \frac{dt}{t}
\frac{1}{I_r(p)}$ gives, by identity (\ref{id1}),
$$ \sum_{i=1}^n x_i^p \leq \sum_{i=1}^n y_i^p  \ , \ \ \ \ ( 0<p<1). $$
Taking $p$th roots,
we obtain the first case of the theorem, since the inequalities
$||x||_p \leq ||y||_p$ extend to the endpoint cases $p=0,1$
automatically by continuity in $p.$
Next, if in addition $\sum_i x_i  = \sum_i y_i \ ,$ then
$\sum_i x_it  = \sum_i y_it \ $ for all $t\geq 0$. Subtracting from this
the inequality $ \log f_r(x,t) \leq \log f_r(y,t)$, one obtains
\begin{equation} \label{sum}
 \sum_i \left( x_it - \log(1 + x_it + \dots + \frac{(x_it)^r}{r!}) \right) \
\geq \  \sum_i \left( y_it - \log(1 + y_it + \dots + \frac{(y_it)^r}{r!}) \right).
\end{equation}
Consider the function $\delta_r(s):= s-\log(1 + s + \dots + \frac{s^r}{r!})$
for $s\geq 0.$
We have $\delta_r(s) \geq s- \log(e^s) =0$ for $s\geq 0.$ When $s\to +\infty,$ we
have $\delta_r(s) = O(s) + O(\log (s^r) ) = O(s).$ When $s \to 0^{+}$ we have
$\delta_r(s) = s - \log\left(e^s-O(s^{r+1})\right) =
s-\log\left(e^s(1-e^{-s}O(s^{r+1})\right) =
s-\log(e^s)-\log\left(1-e^{-s}O(s^{r+1})\right) = O(e^{-s}O(s^{r+1}))
=O(s^{r+1}).$ It follows that the  integrals
$$J_r(p)  := \int_0^\infty
\left(s-\log(1 + s + \dots + \frac{s^r}{r!})\right)
\ s^{-p} \  \frac{ds}{s}$$
are finite (and positive) for all real $p$ in the interval $1<p<r+1.$
Replacing $s$
by $at$ gives the  new identity
\begin{equation}
\label{id2}
\frac{1}{J_r(p)}\int_0^\infty
 \left(at-\log(1 + at + \dots + \frac{(at)^r}{r!})\right)
\ t^{-p} \  \frac{dt}{t} \ = \ a^p \ \ \  \ (a\geq0,\ 1<p<r+1).
\end{equation}
Thus, when $1<p<r+1$ we may integrate (\ref{sum}) with respect to
$t^{-p} \  \frac{dt}{t}
\frac{1}{J_r(p)}$ and use (\ref{id2}) to obtain
$$ \sum_i x_i^p \geq \sum_i y_i^p  \ , \ \ \ \ ( 1<p<r+1). $$
By continuity in $p,$ we obtain $||x||_p \geq ||y||_p$ for
$1\leq p \leq r+1.$
\end{proof}

\noindent {\bf Remarks.}
{\bf (1.1)} If $x \prec_F y$ and $\sum_i x_i  = \sum_i y_i \ ,$
Theorem 1 shows that
$ ||x||_p \geq ||y||_p \ ,\  1 \leq p \leq \infty, (p \in \R).$
{\bf (1.2)} The endpoint cases $p=0,1,r+1$ can also be deduced directly from
the hypotheses instead of by continuity in $p.$ To see this note
that $||x||_0^{nr} = (x_1\dots x_n)^r = (r!)^nF_{rn,r}(x), \ ||x||_1^r =
E_1(x)^r = r!F_{r,r}(x), \ ||x||_{r+1}^{r+1} =
E_1(x)^{r+1} - (r+1)!F_{r+1,r}(x).$\\

\noindent {\Large {\bf 4. Schur-concavity of the $ F_{k,r}(x)$ and related polynomials.} }
\medskip

 For a symmetric polynomial
$\Phi$ with positive coefficients, Schur-concavity on $\R_+^n$
is equivalent to the property $x \prec^w y \Rightarrow
\Phi(x) \leq \Phi(y) $ \cite[Theorem 3.A.8]{MO}, and by
(\ref{S-O}) reduces to checking that
\begin{equation}
\label{Sch}
\bigg(\frac{\partial \Phi}{\partial x_i} -
\frac{\partial \Phi}{\partial x_j}\bigg)/(x_j -x_i)
\ \geq  \ 0 \  \ \ \ \forall x \in \R_{+}^n\ , \ ( x_i\neq x_j).
\end{equation}
To see that this is true for each $\Phi = F_{k,r}\ ,$ one can consider
at once the whole generating function $f_r(x,t)$ (\ref{gen1}) and compute
$\frac{\partial f_{r}(x,t)}{\partial x_i} -
\frac{\partial f_r(x,t)}{\partial x_j}$. One easily obtains
$(x_j-x_i)$ times a new polynomial in $(x,t)$
with {\it positive coefficients }. This shows simultaneously that all the
coefficients $F_{k,r}(x)$ of the generating function
$f_{r}(x,t)$ satisfy (\ref{Sch}) and hence are Schur-concave.
We omit the details since this fact is part of a special case of results of
Proschan and Sethuraman \cite[Theorem 3.J.2, Example 3.J.2.b]{MO}.
The part that we need may be stated as follows.
 Define
$I_k := \{ p \in \Z^n \ |\ p_i \geq 0, \
 \sum p_i = k \}.$
  A subset $S \subset I_k$ is said
 to be Schur-concave if its indicator function
 ${\bf 1}_S$ is Schur-concave on $I_k \ ,$
 that is if
 $p \succ q  \Rightarrow  {\bf 1}_S(p) \leq {\bf 1}_S(q),$
 or equivalently
 $$ p \in S, \ q \in I_k,  \ p \succ q
\Rightarrow  q \in S. $$
\begin{prop} [Proschan and Sethuraman].
Let $k,n \geq 1$ and let $S \subset I_k$ be Schur-concave.
Define the polynomial $H_S$ by
\begin{equation}
\label{HS}
H_S(x_1,\dots,x_n)\  = \ \ \sum_{p \in S}
\ \ \prod_{i=1}^n \frac{x_i^{p_i}}{p_i!}\ .
\end{equation}
Then for all $x,y \in \R_{+}^n \ , $
$ x \succ y
\Rightarrow  H_S(x) \leq H_S(y), $ that is, $H_S$ is Schur-concave
on $\R_{+}^n.$
\end{prop}
To prove just this proposition, the proof of the general theorem
of Proschan and Sethuraman \cite[Theorem 3.J.2]{MO}
can be simplified slightly by directly computing (\ref{Sch})
with $\Phi = H_S$ ;
one obtains a polynomial with positive coefficients.
One example of
a Schur-concave $S \subset I_k$ is the set $S=S_{k,r}:= \{ p \in \Z^n \ |\ p_i \geq 0, \
 \sum p_i = k , \  \max p_i \leq r \}$ for fixed integers $n,k,r \geq 1,$
 as is easily verified. By definition (\ref{Fkr}) we have $F_{k,r} = H_S$
 with $S=S_{k,r}$ ,
 so that the Schur concavity of $F_{k,r}$ is a special case of the
 proposition.
Another example is the set $S=T_{k,r}$ defined by
$$ T_{k,r} :=
\{ \ p \in \Z^n \ |\ p_i \geq 0,
\ \sum p_i = k , \ \ p \
\ {\rm has \ at \ least \ } r
{\rm \ nonzero \ entries} \  \}.
$$
Hence, the corresponding polynomials $H_S= H_{T_{k,r}}=:G_{k,r}$
are Schur-concave. This example can be used to
characterize the majorization relation on $\R_+^n$
as follows.
\begin{prop}
Let $x,y \in \R_{+}^n  $ with $ \sum x_i = \sum y_i $ and suppose that
$G_{k,r}(x) \leq G_{k,r}(y) $ for all integers $k,r \geq 1.$
Then $ x \succ y . $ {\rm (}Combining this with the fact that
the $G_{k,r}(x)$ are Schur-concave, we have the characterization
$ x \succ y  \ \Leftrightarrow \ \sum x_i = \sum y_i$ and $
 \forall \ k,r \geq 1, \ G_{k,r}(x) \leq G_{k,r}(y).${\rm )}
\end{prop}
\begin{proof}
Given $x \in \R_{+}^n$
and a fixed $1\leq r \leq n,$ we may ``compute" the
sum $s_r(x):=\sum_{i=1}^r \ x_i^*$
by first noting that it is the maximum of all possible sums of $r$
entries of $x,$ and then computing this maximum by
using integer $k$-norms as $k \to \infty$ :
$$
s_r(x) = \lim_{k\to \infty} \
\left( \sum_{1\leq i_1 < \dots < i_r \leq n}
\ (x_{i_1} + \dots + x_{i_r})^k \ \right)^{\frac{1}{k}}.
$$
This leads us to consider the symmetric polynomials
\begin{equation}
\label{Mkr}
M_{k,r}(x):= \sum_{1\leq i_1 < \dots < i_r \leq n}
\ (x_{i_1} + \dots + x_{i_r})^k  \  \ \ \ \ \ \ \ (k \in \N).
\end{equation}
By the preceding remarks, the conditions $M_{k,r}(x) \geq M_{k,r}(y) \ \forall k, r$
imply $x \succ y$, for any $x,y \in \R_{+}^n  $ with $ \sum x_i = \sum y_i .$
It now remains to relate the polynomials $M_{k,r}(x)$ to the
$G_{k,r}(x)$ of the proposition.
Consider the polynomials $\overline{G_{k,r}}$ defined by
$$ \overline{G_{k,r}}(x) \ := \
\frac{1}{k!}(x_1 + \dots + x_n)^k - G_{k,r}(x) \ ,
$$
which may be thought of as the sum of all monomial terms in the
expansion of $\frac{1}{k!}(x_1 + \dots + x_n)^k$ containing
{\it less than} $r$ distinct $x_i$ as factors.
To prove the proposition, it suffices to show that when
$k \geq r,$ each $M_{k,r}$
is a linear combination, with positive coefficients, of some of the
$\overline{G_{k,r}}$.
Let
$$\Delta\overline{G_{k,r}}(x)=\overline{G_{k,r+1}}(x) - \overline{G_{k,r}}(x)$$
i.e. the sum of all terms containing
{\it exactly} $r$ distinct $x_i$ as factors.
An expansion of each power in $M_{k,r}$ by the multinomial
theorem gives
$$\frac{1}{k!}M_{k,r} =
{\small \bigg(\begin{array}{c}
n-r  \\
0
 \end{array}\bigg) }
\Delta\overline{G_{k,r}} \ + \
{\small \bigg(\begin{array}{c}
n-r+1  \\
1
 \end{array}\bigg) }
\Delta\overline{G_{k,r-1}}
\ + \dots + \
{\small \bigg(\begin{array}{c}
n-1  \\
r-1
 \end{array}\bigg) }
\Delta\overline{G_{k,1}} \ \ .
$$
Since these binomial coefficients are increasing from left to right,
the result follows after a summation by parts. In fact,
by Pascal's identity we obtain the explicit formula
$$\frac{1}{k!}M_{k,r} =
{\small \bigg(\begin{array}{c}
n-r-1  \\
0
 \end{array}\bigg) }
\overline{G_{k,r+1}} \ + \
{\small \bigg(\begin{array}{c}
n-r  \\
1
 \end{array}\bigg) }
\overline{G_{k,r}}
\ + \dots + \
{\small \bigg(\begin{array}{c}
n-2  \\
r-1
 \end{array}\bigg) }
\overline{G_{k,2}} \ \ .
$$
\end{proof}
Remark: The polynomials $M= M_{k,r}$ (\ref{Mkr}) used in the proof
are Schur-convex, as is easily verified by checking that
$\bigg(\frac{\partial M}{\partial x_i} -
\frac{\partial M}{\partial x_j}\bigg)/(x_i -x_j)
\ \geq  \ 0 \  \forall x \in \R_{+}^n$.
Hence, majorization can be characterized using
these polynomials as well. That is, for $x,y \in \R_+^n$ with
$ \sum x_i = \sum y_i $, we have
$x \succ y \Leftrightarrow M_{k,r}(x) \geq M_{k,r}(y) \ \forall k, r$.

In contrast to the $G_{k,r}$, as noted in \S2.3, the conditions
$F_{k,r}(x) \leq F_{k,r}(y) \ \forall \ k,r \geq 1$
(i.e. $x \prec_F y $)
are not strong enough to imply majorization
when $\sum x_i = \sum y_i$ (if $n\geq 3$).
Nevertheless, it will be shown that
$x \prec_F y $ is strong enough to imply $x \prec_L y $
(Theorem 9 in the next section).\\

\noindent {\Large {\bf 5. The relation $ x \prec_L y $ and
$  \sup_{||z||_1 =1} \ \prod_i (z_i+ t x_i)$.} }
\bigskip

\noindent  {\bf 5.1. Functionals related to majorization.}
\medskip

We will discuss certain functionals which
lead naturally to the relation $\prec_L$, in several equivalent
forms.
These functionals were introduced in \cite[\S 3.2]{K1}
in connection with the conjectures of \S2.2, in an attempt to
formulate a weaker version of the majorization relation suited
to those conjectures. We begin in the context
of positive functions $X, Y$ on $[0,2\pi]$, which in the application
to the questions of \S2.2 would be given by $X= |D_N|^2$ and
$Y=|f|^2$. This is not significant
and we will soon switch to analogous considerations with
vectors $x,y \in \R_{+}^n$.
The initial idea is that the majorization relation
itself can be stated in terms of a functional $J$ as follows.
If $X:[0,2\pi] \to [0,\infty)$ is a bounded measurable function,
define
\begin{equation}
\label{Jinf}
J(X,t):=\inf_{||g||_1 \leq t} \quad
\int_0^{2\pi} |1-g(\theta)| X(\theta)
d\theta/2\pi, \quad (t \geq 0),
\end{equation}
where the norm is defined by $||g||_1 := \int_0^{2\pi} |g(\theta)|
d\theta/2\pi $. It is easily seen that
$$
J(X,t)=\inf_{|E| \leq 2\pi t} \quad
\int_0^{2\pi} (1-{\bf 1}_E(\theta)) X(\theta)
d\theta/2\pi
=\int_{2\pi t}^{2\pi} X^*(\theta)
d\theta / 2\pi
$$
where $X^*$ is the decreasing rearrangement of $X$. Hence,
for two functions
$X,Y\geq 0$,
the relation $X \prec^w Y$ (see \S1) is equivalent
to the property
\begin{equation}
\label{J}
J(X,t) \leq J(Y,t) \ \ \forall t \geq 0.
\end{equation}
The latter is thus also equivalent to the usual majorization relation  $X \succ Y$
when $\int_0^{2\pi}X =\int_0^{2\pi}Y$.
Moreover, if one examines the ``standard" proofs \cite{HLP} of the equivalence
\begin{equation}
\label{maj}
X \succ Y \ \ \ \Leftrightarrow \ \ \
\int_0^{2\pi} (X-\lambda)_+ \ d\theta \geq
\int_0^{2\pi} (Y-\lambda)_+ \ d\theta, \quad \forall \lambda \geq 0,
\end{equation}
(assuming $\int_0^{2\pi}X =\int_0^{2\pi}Y$),
one sees that these proofs amount to considering the Legendre transform
(``Fenchel conjugate", ``Young transform", etc.)
\begin{equation}
\label{L1}
L(X, \lambda) :=
\inf_{ t \geq 0} \ \left( J(X,t) + \lambda t \right)\ , \ \ \ \lambda \geq 0.
\end{equation}
It is well known that such a transform can be inverted by
another similar transform, in this case
\begin{equation}
\label{L2}
\sup_{ \lambda \geq 0} \ \left( L(X,\lambda) - \lambda t \right)\ \
=  \ J(X,t) \ , \ \ \ t \geq 0,
\end{equation}
if, for instance, the input function $J(X,t)$ is a convex, nonincreasing and
nonegative function of $t$ on $[0,\infty)$ \cite[Ch. 4]{Stoer}. (The latter hypotheses
are easily verified in the present example.)
From this point of view, it is the existence of the two relations (\ref{L1})
and (\ref{L2}) which ensures that one obtains both directions
of the equivalence in (\ref{maj}) \cite[\S3.2]{K1}, \cite{B}.
(To be specific: An easy calculation gives
$L(X,\lambda) = \int_{0}^{2\pi} \min(X(\theta), \lambda)
d\theta / 2\pi$. By (\ref{L1}) and (\ref{L2}) it follows that
 condition (\ref{J}) (i.e. the relation $X \prec^w Y$) is equivalent to
 $\int_{0}^{2\pi} \min(X(\theta), \lambda)
d\theta \  \leq \ \int_{0}^{2\pi} \min(Y(\theta), \lambda) d\theta \ \forall
\lambda \geq 0$,
which in turn is clearly equivalent to
$\int_0^{2\pi} (X-\lambda)_+ \ d\theta \geq
\int_0^{2\pi} (Y-\lambda)_+ \ d\theta \quad \forall \lambda \geq 0$
under the extra condition $\int_0^{2\pi}X =\int_0^{2\pi}Y$.)
We remark that in the theory of interpolation of norms \cite{BL},
such Legendre transforms occur implicitly in various duality results,
 as is well known.
For example, the $K$-functional and the $E$-functional are
Legendre transforms of each other \cite[Ch. 7]{BL}, \cite[\S3.4]{Per}.

Once the majorization relation has been reformulated as above,
it is natural to look for variants of it by modifying
(\ref{Jinf}) in some simple way, such as changing the form of the integrand
or the constraints on $g$.
The following modification was studied in \cite[\S 3.2]{K1}:
\begin{equation}
\label{Iinf}
I_1(X,t):= \inf_{||P||_2 \leq t, \ P(0)=0} \quad
\int_0^{2\pi} |1-P(e^{i\theta})|^2 X(\theta)
 \ d\theta/2\pi, \quad (t \geq 0),
\end{equation}
where $P$ ranges over all complex polynomials
with $P(0)=0$, and
$||P||_2 := (\int_0^{2\pi}  |P(e^{i\theta})|^2 d\theta/2\pi)^{1/2}.$
The choice of this modification was essentially guided by systematic guessing in view of
desired applications to the sharp Littlewood conjectures (see \S2.2).
For example, one particular motivation for the analytic factor $1-P(e^{i\theta})= 1-P(z)$ was
the remark that $(1-z)D_N(z) = 1-z^N$, which suggests that in some sense $D_N(z)$ is
 ``annihilated" by such factors
more easily than are the other $f(z)$'s entering in the conjectures.
Next, the functional (\ref{Iinf}) induces a relation $\prec_1$ via the definition
$$ X\prec_1 Y \ \ \  \Leftrightarrow \ \ \ I_1(X,t) \leq I_1(Y,t) \ \ \forall t\geq0.$$
Using the above method of Legendre transforms
\cite[Proposition 3.2.3]{K1}, it can be shown that
\begin{equation}
\label{leg}
X\prec_1 Y \ \ \  \Leftrightarrow \ \ \
\int_0^{2\pi} \log(X(\theta)+ \lambda)
 \ d\theta  \quad \leq \quad
\int_0^{2\pi} \log(Y(\theta)+ \lambda)
 \ d\theta , \quad   \quad (\forall \ \lambda \geq 0).
 \end{equation}
(We review the proof in Lemma 5 and Theorem 6, in the discrete case.)
This shows that $X \prec_1 Y$ is weaker than the weak
version of majorization
$X \prec^w Y$ (see \S1),
and that, as noted in \S2.3, it only implies
 the $L^p$ relations $||X||_p \leq ||Y||_p \ , \
( 0 \leq p \leq 1),$ and
$||X||_p \geq ||Y||_p \ ,\ ( 1 \leq p \leq 2),$
if $||X||_1 = ||Y||_1\ $. Thus we arrive at
the problem of how to further modify our functional (\ref{Iinf}) so that
the latter interval of $L^p$ relations becomes $1 \leq p \leq \infty$
instead of $1 \leq p \leq 2$, but so that the relation induced by the
functional is still weaker than majorization.
 (More precisely, weak enough so that it does {\it not} imply $||X||_p \leq ||Y||_p$
 for some $p < 0$, in view of the multiple-root counter-examples of \S2.2.) We give
such a modification below in (\ref{hyinf}) and (\ref{dhyinf}). The main idea will be that
the ``cause" of the latter restriction $  p \leq 2$ can be traced back to the use of ``2" in
the constraint condition $||P||_2 \leq t$ appearing in (\ref{Iinf}), and that consequently
one needs to try a different kind of constraint condition.
We will then switch to the discrete setting  and prove several equivalent forms
of the relation induced by
(\ref{dhyinf}) (Theorem 7), one of which is the relation
$\prec_L$ defined by (\ref{L}).\\

\noindent  {\bf 5.2.
A modification of $I_1(X,t)$ ; the functional $I_\infty(X,t)$.}
\medskip

We first give a purely ``measure theoretic" formula
for  $I_1(X,t),$ in the following lemma.
That is, there will be no mention of polynomials
or other analytic functions.

\begin{lemma}
For any bounded measurable $X \geq 0$ we have
\begin{equation}
\label{hinf}
I_1(X,t)= \inf_{||h||_2^2 \leq 1+t^2, \ ||h||_0=1} \quad
\int_0^{2\pi} |h(\theta)|^2 X(\theta)
 \ d\theta/2\pi, \quad (t \geq 0),
 \end{equation}
 where $h$ runs over $L^2$ functions and
 $||h||_0 := \exp \left( \int_0^{2\pi} \log|h(\theta)|
 \ d\theta/2\pi \right)$ {\rm(}the geometric mean{\rm)}.
\end{lemma}

\begin{proof} This follows by the ideas in various proofs of Szeg\"o's theorem
 \cite[Ch. 4, Theorem 3.1]{Gar}, \cite[Ch. 2]{Koo}:
Fix $X$ and let $t\geq 0.$
The definition (\ref{Iinf}) can be restated in the form
$$I_1(X,t)= \inf_{||Q||_2^2 \leq 1+t^2, \ Q(0)=1} \quad
\int_0^{2\pi} |Q(e^{i\theta})|^2 X(\theta)
 \ d\theta/2\pi, \quad (t \geq 0), $$
where $Q$ runs over polynomials. Next, we may
restrict $Q$ to polynomials having no roots in the open unit disc
$|z| < 1.$ This can be seen by using Blaschke factors, as follows.
Given any polynomial $Q$ satisfying the constraints
($||Q||_2^2 \leq 1+t^2, \ Q(0)=1$), we may
factor it as $Q(z) = (1-r_1z)\dots(1-r_mz).$
If some factor, say $(1-r_1z),$ has $|r_1| > 1,$ then we
replace $Q$ by the new polynomial
$$Q_1(z) := (1-\frac{1}{\overline{r_1}}z)
(1-r_2z)\dots(1-r_mz) = \frac{1}{\overline{r_1}}
\frac{\overline{r_1}-z}{1-r_1z}
Q(z),$$
which still has $Q_1(0)=1$ and
on the unit circle satisfies $|Q_1(e^{i\theta})|
=\frac{1}{|r_1|}|Q(e^{i\theta})| \leq |Q(e^{i\theta})|.$
Thus $Q_1$ also satisfies the constraint
$||Q_1||_2^2 \leq 1+t^2$ {\it and} it gives
$\int_0^{2\pi} |Q_1(e^{i\theta})|^2 X(\theta)
 \ d\theta/2\pi \leq
 \int_0^{2\pi} |Q(e^{i\theta})|^2 X(\theta)
 \ d\theta/2\pi.$ Proceeding this way, we can move all
 of the undesirable roots of $Q$ to the complement of the open unit disc,
 which proves the claim. Now observe that if a polynomial $Q$ has
 no roots in $|z| < 1,$ the condition $Q(0)=1$ implies
 that $||Q||_0 = 1.$ So, if $\alpha$ denotes the infimum on the
 right hand side of (\ref{hinf}), we have shown that $I_1(X,t) \geq \alpha.$

 To prove $I_1(X,t) \leq \alpha,$ fix a nonnegative $h \in L^2$ with
 $||h||_2^2 \leq 1+t^2, \ ||h||_0=1.$ Then there is an
 ``outer function" $f \in H^2$
 with $|f|=h$ on the unit circle whose constant Fourier coefficient is
 $\hat{f}(n)|_{n=0} = f(z)|_{z=0} =||h||_0=1.$ Taking $Q_n(e^{i\theta})$
 to be the polynomials defined by truncating the Fourier series of
 $f(e^{i\theta})$ at the integers $n,$ we have $Q_n(0) = 1, \ ||Q_n||_2^2 \leq
 ||f||_2^2 = ||h||_2^2 \leq 1+t^2 \ .$ Moreover, as $n\to \infty\ ,$
 $ \int_0^{2\pi} |Q_n(e^{i\theta})|^2 X(\theta)
 \ d\theta/2\pi \to \int_0^{2\pi} |h(e^{i\theta})|^2 X(\theta)
 \ d\theta/2\pi$ since $X$ is bounded and $Q_n \to f$ in $L^2.$
\end{proof}

From (\ref{hinf}) it is evident that the squares serve
no further purpose. In other words,
$$
I_1(X,t)= \inf_{||h||_1 \leq 1+t^2, \ ||h||_0=1} \quad
\int_0^{2\pi} |h(\theta)| X(\theta)
 \ d\theta/2\pi, \quad (t \geq 0).
$$
Moreover, since we are always interested in defining our relations
by conditions of the form $I_1(X,t) \leq I_1(Y,t)$ over {\it all} $t\geq 0$, then a change of parameter
such as $t\to \sqrt{t}$ will not affect the relation
thus defined. Therefore, let us change notation at this point
by re-defining
\begin{equation}
\label{h1inf}
I_1(X,t):= \inf_{||h||_1 \leq 1+t, \ ||h||_0=1} \quad
\int_0^{2\pi} |h(\theta)| X(\theta)
 \ d\theta/2\pi, \quad (t \geq 0).
 \end{equation}
If we now wish to modify $I_1(X,t)$ such that the above mentioned
$p\leq 2$ restriction becomes $p\leq \infty$ (see \S5.1), a natural choice
would seem to be the functional defined by
\begin{equation}
\label{hyinf}
I_\infty(X,t):= \inf_{||h||_\infty \leq 1+t, \ ||h||_0=1} \quad
\int_0^{2\pi} |h(\theta)| X(\theta)
 \ d\theta/2\pi, \quad (t \geq 0).
 \end{equation}
 \noindent This functional is  essentially the $G(t,X)$ defined
 in \cite[eq. (37)]{K1}, except for a slight re-parametrization.
 The present discussion makes its relation to $I_1(X,t)$ more
 transparent.
From this point on we will continue the discussion for the
discrete versions of $I_1$ and $I_\infty$ defined
for $x \in \R_{+}^n$ by
\begin{equation}
\label{dh1inf}
I_1(x,t):= \inf_{h \in \R_{+}^n, \ ||h||_1 \leq 1+t, \ ||h||_0=1} \quad
\frac{1}{n}\sum_{i=1}^n h_i x_i , \quad (t \geq 0),
 \end{equation}
and
\begin{equation}
\label{dhyinf}
I_\infty(x,t):=
\inf_{h \in \R_{+}^n, \ ||h||_\infty \leq 1+t, \ ||h||_0=1} \quad
\frac{1}{n}\sum_{i=1}^n h_i x_i , \quad (t \geq 0),
 \end{equation}
where the $l^p$ norms are with respect to normalized
counting measure; thus $||h||_1 = \frac{1}{n}\sum_i |h_i|\ , \
||h||_0 = (\prod_i |h_i| )^{1/n} .$
Define corresponding relations on $\R_{+}^n$ by
\begin{equation}
\label{dh1}
 x\prec_1 y \ \Leftrightarrow \ I_1(x,t) \leq I_1(y,t) \ \ \forall t\geq0,
 \end{equation}
and
\begin{equation}
\label{dhy}
 x\prec_\infty y \
\Leftrightarrow \  I_\infty(x,t) \leq I_\infty(y,t) \ \ \forall t\geq0.
\end{equation}
\medskip

\noindent  {\bf 5.3.
Equivalence of $x \prec_\infty y$ and $x \prec_L y$.}
\medskip

\begin{lemma} For any $x \in \R_{+}^n,$
$$\inf_{ h \in \R_{+}^n, \ ||h||_0=1 } \
\frac{1}{n}\sum_{i=1}^n h_i x_i  \
= \ \left(\prod_{i=1}^n x_i \right)^{1/n}  =:||x||_0 \ .$$

\end{lemma}
\begin{proof} Exercise.
\end{proof}
\noindent
We now give the following characterization of $\prec_1$
for later comparison with $\prec_\infty \ .$ This is just
the discrete version of the result (\ref{leg})
already mentioned, and the proof will be similar.

\begin{theorem} Let $x,y \in \R_{+}^n.$ Then
$$x\prec_1 y \ \ \ \Leftrightarrow \ \ \
\prod_{i=1}^n (x_i+\lambda) \   \leq \
  \prod_{i=1}^n (y_i+\lambda) \ ,  \  \forall \lambda \geq 0 .$$
\end{theorem}

\begin{proof} It is easy to check that $I_1(x,t)$ is a
convex, nonincreasing and nonegative function of $t$ on $[0,\infty)$.
Hence, as above in (\ref{L1}) and (\ref{L2}), the Legendre transform defined by
\begin{equation}
\label{inf1}
L(x, \lambda) :=
\inf_{ t \geq 0} \ \left( I_1(x,t) + \lambda t \right)\ , \ \ \ \lambda \geq 0,
\end{equation}
can be inverted via
\begin{equation}
\label{sup1}
\sup_{ \lambda \geq 0} \ \left( L(x,\lambda) - \lambda t \right)\ \
=  \ I_1(x,t) \ , \ \ \ t \geq 0.
\end{equation}
Next, it is clear from (\ref{inf1}) and (\ref{sup1}) that
$I_1(x,t) \leq I_1(y,t) \ \forall t\geq 0$ if and only if
$L(x,\lambda) \leq L(y,\lambda) \ \forall \lambda \geq 0.$
It thus remains to compute $L(x, \lambda)$ explicitly.
In the following it is understood that $h \in \R_{+}^n.$  \\

$L(x,\lambda) =
{\displaystyle \inf_{ t \geq 0} }\ \left(
{\displaystyle\inf_{ ||h||_1 \leq 1+t, \ ||h||_0=1}}
\left(\frac{1}{n}\sum_{i=1}^n h_i x_i \right) \
 +  \ \lambda t \right)
$ \\

$=
{\displaystyle\inf_{ t \geq 0}} \ \left(
{\displaystyle\inf_{ ||h||_1 \leq 1+t, \ ||h||_0=1}}
\left(\frac{1}{n}\sum_{i=1}^n h_i x_i \
 +  \ \lambda t \right) \ \right) \ \ \ \ \ \ \ \ (=:a)
$ \\

$ = {\displaystyle\inf_{  ||h||_0=1}}
\left(\frac{1}{n}\sum_{i=1}^n h_i x_i
 \  +  \ \lambda (||h||_1-1) \right) \ \ \ \ \ \ \ \ (=: b)
$ \\

\noindent To see the last equality (that $a=b $) note first that
clearly $a \geq b,$ since
we may make the substitution $t \geq ||h||_1-1.$
On the other hand, given any $h \in \R_{+}^n$
with $||h||_0=1,$ note that $||h||_0 \leq ||h||_1\ ,$
so defining $s:=||h||_1-1$ gives $s\geq 0, \  ||h||_1 \leq 1+s,$
and thus $\frac{1}{n}\sum_{i=1}^n h_i x_i
 \  +  \ \lambda (||h||_1-1) = \frac{1}{n}\sum_{i=1}^n h_i x_i \
 +  \ \lambda s \geq a $ by definition of $\inf.$ Hence $b
 \geq a,$ and thus $a = b.$ Continuing the above computation, we have \\

 $ b = {\displaystyle\inf_{  ||h||_0=1}}
\left(\frac{1}{n}\sum_{i=1}^n h_i x_i
 \  +  \ \lambda\frac{1}{n}\sum_{i=1}^n h_i \right) \ \ -\lambda$ \\

 $= {\displaystyle\inf_{  ||h||_0=1}}
\left(\frac{1}{n}\sum_{i=1}^n h_i (x_i +\lambda)
  \right) \ \ -\lambda$ \\

$ = \left(\prod_{i=1}^n (x_i+\lambda) \right)^{1/n} \ -\lambda$ \\

\noindent
by Lemma 5. The theorem follows.
\end{proof}

\noindent Theorem 6 provides a connection between $\prec_1$
and the partial order $\prec_E$ of \S2.3 : We have
immediately that $x\prec_E y  $ implies $x\prec_1 y  $.
There is an analogous connection
between $\prec_\infty$ and $\prec_F\ $:
In Theorem 9 we will see that $x\prec_F y  $ implies $x\prec_\infty y  .$
But first let us connect $\prec_\infty$ with $\prec_L$.

\begin{theorem}
Let $x,y \in \R_{+}^n.$
The following three conditions are equivalent.

{\rm (a)} $ \  x\prec_\infty y \ \ $ {\rm [see (\ref{dhy}) ].}\\

{\rm (b)} $ {\displaystyle \ \sup_{||z||_1 =1} \ \ \prod_{i=1}^n (z_i + tx_i) \ \  \leq \
\ \sup_{||z||_1 =1} \ \ \prod_{i=1}^n (z_i + ty_i) \ , \ \  \forall t > 0 .}$ \\

{\rm (c)} $  \  x\prec_L y \ \ $ {\rm [see (\ref{L}) ].}
\end{theorem}
\begin{proof} We first prove that (a) and (b) are equivalent.
The method is the same as in the proof of Theorem 6.
It is easily checked that $I_\infty(x,t)$ is a convex, nonincreasing and
nonegative function of $t$ on $[0,\infty)$. Hence, if we define
a Legendre transform of $I_\infty$ by
\begin{equation}
\label{inf2}
L(x, \lambda) :=
\inf_{ t \geq 0} \ \left( I_\infty(x,t) + \lambda t \right)\ , \ \ \ \lambda \geq 0,
\end{equation}
it can be inverted as before in (\ref{sup1}):
\begin{equation}
\label{sup2}
\sup_{ \lambda \geq 0} \ \left( L(x,\lambda) - \lambda t \right)\ \
=  \ I_\infty(x,t) \ , \ \ \ t \geq 0.
\end{equation}
Thus we immediately have the equivalence
$I_\infty(x,t) \leq I_\infty(y,t) \ \forall t\geq 0 \ \ \Leftrightarrow \ \
L(x,\lambda) \leq L(y,\lambda) \ \forall \lambda \geq 0.$
We now compute $L(x,\lambda).$ As before, vectors
such as $h,x,y, z$ are assumed
to be in  $\R_{+}^n .$  \\

$L(x,\lambda) =
{\displaystyle \inf_{ t \geq 0} }\ \left(
{\displaystyle\inf_{ ||h||_\infty \leq 1+t, \ ||h||_0=1}}
\left(\frac{1}{n}\sum_{i=1}^n h_i x_i \right) \
 +  \ \lambda t \right)
$ \\

$=
{\displaystyle\inf_{ t \geq 0}} \ \left(
{\displaystyle\inf_{ ||h||_\infty \leq 1+t, \ ||h||_0=1}}
\left(\frac{1}{n}\sum_{i=1}^n h_i x_i \
 +  \ \lambda t \right) \ \right) \ \ \ \ \ \ \ \ (=: a)
$ \\

$ = {\displaystyle\inf_{  ||h||_0=1}}
\left(\frac{1}{n}\sum_{i=1}^n h_i x_i
 \  +  \ \lambda (||h||_\infty-1) \right) \ \ \ \ \ \ \ \ (=: b),
$ \\

\noindent where the proof that $a = b$ is similar
to the corresponding step in the proof of Theorem 6. Next \\

 $ b = {\displaystyle\inf_{  ||h||_0=1}}
\left(\frac{1}{n}\sum_{i=1}^n h_i x_i
 \  +  \ \lambda ||h||_\infty \right) \ \ -\lambda$ \\

$= {\displaystyle\inf_{  ||h||_0=1}}
\left(\frac{1}{n}\sum_{i=1}^n h_i x_i
 \  +  \ \lambda
 {\displaystyle\sup_{  ||z||_1=1}}
 \frac{1}{n}\sum_{i=1}^n h_i z_i
  \right) \ \ -\lambda$ \\

 $= {\displaystyle\inf_{  ||h||_0=1}}
\left(
{\displaystyle\sup_{  ||z||_1=1}}
\left(
\frac{1}{n}\sum_{i=1}^n h_i x_i
 \  +  \ \lambda
 \frac{1}{n}\sum_{i=1}^n h_i z_i
 \right)
  \right) \ \ -\lambda$ \\

$= {\displaystyle\inf_{  ||h||_0=1}}
\left(
{\displaystyle\sup_{  ||z||_1=1}}
\left(
\frac{1}{n}\sum_{i=1}^n h_i (x_i
 \  +  \ \lambda
  z_i)
 \right)
  \right) \ \ -\lambda  \ \ \ \ \ \ \ \ (=: a')$ \\

$= {\displaystyle\sup_{  ||z||_1=1}}
\left(
{\displaystyle\inf_{  ||h||_0=1}}
\left(
\frac{1}{n}\sum_{i=1}^n h_i (x_i
 \  +  \ \lambda
  z_i)
 \right)
  \right) \ \ -\lambda  \ \ \ \ \ \ \ \ (=: b')$ \\

$= {\displaystyle\sup_{  ||z||_1=1}}
 \ \left(\prod_i (x_i+\lambda z_i)  \right)^{1/n} \ \ -\lambda$ \\

\noindent where the equality $a' =b'$ is a consequence of
von Neuman's min-max theorem (see \cite[Theorem 4.2]{Sion}), and the last equality follows by Lemma 5. This completes the proof
of (a) $\Leftrightarrow$ (b). (Alternatively, one can easily
evaluate the supremum in (b), as well as find explicitly the saddle point
in the latter min-max problem, and
thus avoid appealing to a min-max theorem in the proof that $a' =b'$. This will be discussed
in Remarks (7.1) and (7.2) after the proof.)

We now prove (b) $\Leftrightarrow$ (c). First note that (b) may be
stated in the form
$$ {\displaystyle\sup_{  ||z||_1=\lambda}}
\ \  \frac{1}{n}\sum_i \log(x_i+z_i)
\ \ \leq \ \
 {\displaystyle\sup_{  ||z||_1=\lambda}}
 \ \ \frac{1}{n}\sum_i \log(y_i+z_i) , \ \ \  \ \ \ \forall \lambda > 0.$$
It is easy to check that the function
$$ F(x,\lambda) :={\displaystyle\sup_{  ||z||_1=\lambda}}
 \ \frac{1}{n}\sum_i \log(x_i+z_i), \ \ \  \ \ \ \lambda > 0, $$
 is concave and nondecreasing in $\lambda.$ Thus the following
 Legendre transform
$$ \mathcal{L}(x,t) :={\displaystyle\sup_{  \lambda > 0}}
 \ \left(F(x,\lambda) - \lambda t \right), \ \ \  \ \ \ t > 0, $$
 is well defined and may be inverted via
 $$ F(x,\lambda) ={\displaystyle\inf_{  t > 0}}
 \ \left(\mathcal{L}(x,t) + \lambda t \right),
 \ \ \  \ \ \ \lambda > 0. $$
 Hence, (b) is equivalent to $\mathcal{L}(x,t) \leq \mathcal{L}(y,t),
 \ \forall t > 0.$ We now compute $\mathcal{L}(x,t).$ \\

 $ \mathcal{L}(x,t) ={\displaystyle\sup_{  \lambda > 0}}
 \ \left({\displaystyle\sup_{  ||z||_1=\lambda}}
 \ \frac{1}{n}\sum_i \log(x_i+z_i)
  - \lambda t \right) $\\

$={\displaystyle\sup_{  \lambda > 0}}
 \ \left({\displaystyle\sup_{  ||z||_1=\lambda}}
 \ \left(\frac{1}{n}\sum_i \log(x_i+z_i)
  - ||z||_1 t\right) \right) $\\

$={\displaystyle\sup_{ z \in \R_{+}^n}}
 \ \left(\frac{1}{n}\sum_i \log(x_i+z_i)
  - ||z||_1 t \right) $\\

$={\displaystyle\sup_{ z \in \R_{+}^n}}
 \ \frac{1}{n}\sum_i  \left(\log(x_i+z_i)
  -  tz_i \right). $\\

\noindent
To compute the latter supremum, we simply choose each $z_i$ to
maximize the expression $\log(x_i+z_i) -  tz_i .$ For
$t>0,$ this is easily seen to occur when $z_i = (\frac{1}{t}-x_i)_+ \ ,$  giving\\

$\log(x_i+z_i) -  tz_i = \log(x_i+(\frac{1}{t}-x_i)_+)
-  t(\frac{1}{t}-x_i)_+ =
\max(\log(\frac{1}{t}), \log(x_i)) - (1-tx_i)_+ $\\

$= \min(tx_i,1) +\log_+(tx_i) -(1+\log t).$\\

$= t\left(\min(x_i,\frac{1}{t}) +
\frac{1}{t}\log_+(tx_i) \right) -(1+\log t).$\\

\noindent Thus
$$\mathcal{L}(x,t) = \frac{1}{n}t\sum_i
\left(\min(x_i,\frac{1}{t}) +
\frac{1}{t}\log_+(tx_i) \right) \ - \ (1+\log t) \ , \ \ \ \ \ \ t>0.$$
Hence, the relation $ \mathcal{L}(x,t) \leq \mathcal{L}(y,t) \ \forall
t > 0 $ is equivalent to the condition
$$\sum_i
\left(\min(x_i,\frac{1}{t}) +
\frac{1}{t}\log_+(tx_i) \right) \ \ \leq \ \
\sum_i
\left(\min(y_i,\frac{1}{t}) +
\frac{1}{t}\log_+(ty_i) \right), \ \ \ \ \ \ \forall
t > 0 $$
which is seen to be equivalent to (c) of the theorem upon identifying
$\frac{1}{t} $ with the parameter $\lambda$ in (\ref{psib}).  This completes the
proof of (b) $\Leftrightarrow$ (c).
\end{proof}

%%%%%%%%%%%%%%%%%%%%%%%%%%%%%%%%%%%%%%%%%%%%%%%%%%%%%
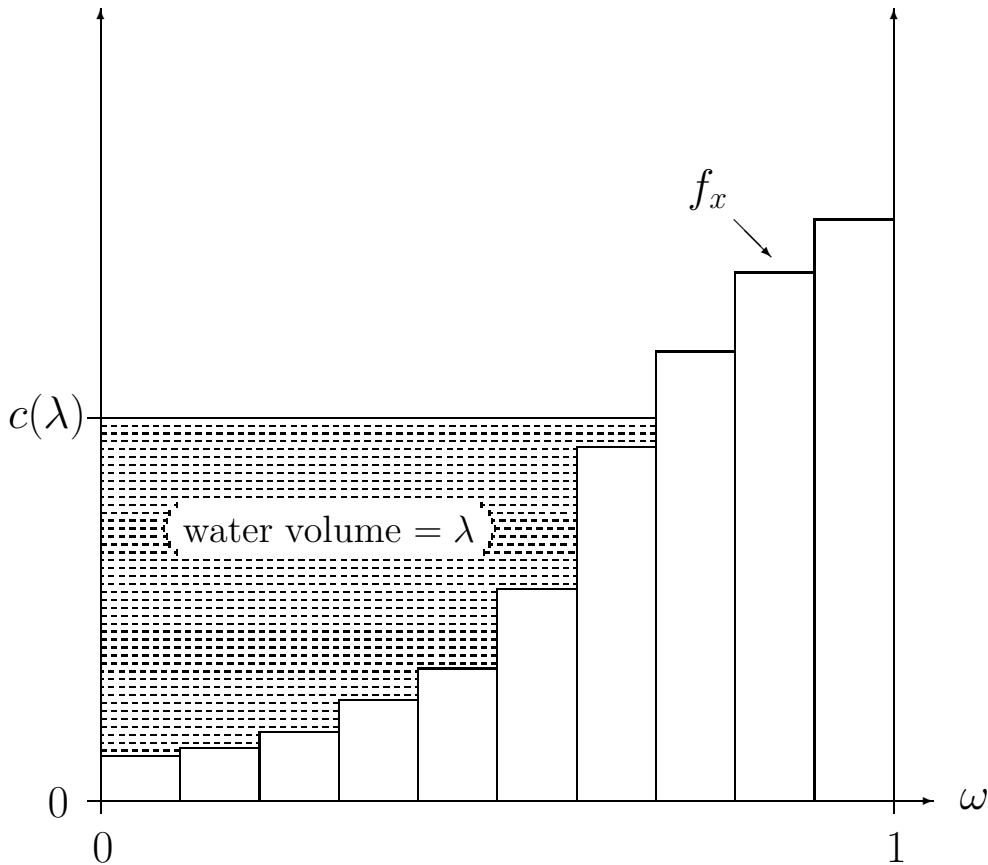
\begin{figure}[h] 
\begin{picture}(460,350)(-85,-15)

\multiput(0,-5)(0,0){1}{\line(0,1){300}}%y-axis
\multiput(0,295)(0,0){1}{\vector(0,1){5}}%y-axis arrow
\multiput(-20,-6)(0,0){1}{{\LARGE $0$}}% y=0 label

\multiput(-5,0)(0,0){1}{\line(1,0){320}} %\omega-axis
\multiput(295,0)(0,0){1}{\vector(1,0){20}} %\omega-axis arrow
\multiput(325,-3)(0,0){1}{\LARGE $\omega$}%\omega label
\multiput(-3,-23)(0,0){1}{\LARGE $0$}%\omega=0 label
\multiput(297,-23)(0,0){1}{\LARGE $1$}%\omega=1 label
\multiput(300,295)(0,0){1}{\vector(0,1){5}}%\omega=1 arrow

\multiput(300,-5)(0,0){1}{\line(0,1){300}}

% draw graph of f_x(\omega)
\multiput(0,17)(0,0){1}{\line(1,0){30}} %0-1
\multiput(30,16.8)(0,0){1}{\line(0,1){3.4}} %v1
\multiput(0,19)(0,3){1}{\dashbox{2}(30,0){} } %water
\multiput(30,20)(0,0){1}{\line(1,0){30}} %1-2
\multiput(60,19.8)(0,0){1}{\line(0,1){6}} %v2
\multiput(0,22)(0,3){2}{\dashbox{2}(60,0){} } %water
\multiput(60,26)(0,0){1}{\line(1,0){30}} %2-3
\multiput(90,25.8)(0,0){1}{\line(0,1){12.4}} %v3
\multiput(0,28)(0,3){4}{\dashbox{2}(90,0){} } %water
\multiput(90,38)(0,0){1}{\line(1,0){30}} %3-4
\multiput(120,37.8)(0,0){1}{\line(0,1){12.4}} %v4
\multiput(0,40)(0,3){4}{\dashbox{2}(120,0){} } %water
\multiput(120,50)(0,0){1}{\line(1,0){30}} %4-5
\multiput(150,49.8)(0,0){1}{\line(0,1){30.4}} %v5

\multiput(0,52)(0,3){10}{\dashbox{2}(150,0){} } %water

\multiput(150,80)(0,0){1}{\line(1,0){30}} %5-6
\multiput(180,79.8)(0,0){1}{\line(0,1){54.4}} %v6
\multiput(0,82)(0,3){4}{\dashbox{2}(180,0){} } %water

\multiput(222,227)(0,0){1}{\LARGE $f_x$}% f_x graph label
\multiput(236,220)(0,0){1}{ \vector(1,-1){14}}% f_x(\omega) graph label

\multiput(0,94)(0,0){1}{\dashbox{2}(29,0){} } %water
\multiput(0,97)(0,0){1}{\dashbox{2}(26,0){} } %water
\multiput(0,100)(0,3){1}{\dashbox{2}(25,0){} } %water
\multiput(0,103)(0,3){1}{\dashbox{2}(24,0){} } %water
\multiput(0,106)(0,3){1}{\dashbox{2}(25,0){} } %water
\multiput(0,109)(0,0){1}{\dashbox{2}(26,0){} } %water
\multiput(0,112)(0,0){1}{\dashbox{2}(29,0){} } %water

\multiput(145,94)(0,0){1}{\dashbox{2}(35,0){} } %water
\multiput(147,97)(0,0){1}{\dashbox{2}(33,0){} } %water
\multiput(148,100)(0,1){1}{\dashbox{2}(32,0){} } %water
\multiput(149,103)(0,1){1}{\dashbox{2}(31,0){} } %water
\multiput(148,106)(0,1){1}{\dashbox{2}(32,0){} } %water
\multiput(147,109)(0,0){1}{\dashbox{2}(33,0){} } %water
\multiput(144,112)(0,0){1}{\dashbox{2}(36,0){} } %water

\multiput(0,115)(0,3){7}{\dashbox{2}(180,0){} } %water

\multiput(180,134)(0,0){1}{\line(1,0){30}} %6-7
\multiput(210,134.8)(0,0){1}{\line(0,1){35.4}} %v7
\multiput(0,136)(0,3){3}{\dashbox{2}(210,0){} } %water
\multiput(-5,145)(0,0){1}{\line(1,0){215} } %water line-------------
\multiput(-35,141)(0,0){1}{\LARGE $c(\lambda)$}% y=c label

\multiput(210,170)(0,0){1}{\line(1,0){30}} %7-8
\multiput(240,169.8)(0,0){1}{\line(0,1){30.4}} %v8
\multiput(240,200)(0,0){1}{\line(1,0){30}} %8-9
\multiput(270,199.8)(0,0){1}{\line(0,1){20.4}} %v9
\multiput(270,220)(0,0){1}{\line(1,0){30}} %9-10

\multiput(31,98)(0,0){1}{\Large water volume $= \lambda$}

%%%%%%%%%%%% vertical lines (optional)

\multiput(30,0)(0,0){1}{\line(0,1){20.2}} %v1

\multiput(60,0)(0,0){1}{\line(0,1){26.2}} %v2

\multiput(90,0)(0,0){1}{\line(0,1){38.0}} %v3

\multiput(120,0)(0,0){1}{\line(0,1){50.2}} %v4

\multiput(150,0)(0,0){1}{\line(0,1){80.2}} %v5

\multiput(180,0)(0,0){1}{\line(0,1){134.0}} %v6
\multiput(210,0)(0,0){1}{\line(0,1){170.0}} %v7
\multiput(240,0)(0,0){1}{\line(0,1){200.0}} %v8
\multiput(270,0)(0,0){1}{\line(0,1){220.0}} %v9

%%%%%%%%%%%%%%%%%%%%%%%%%%%%%%%%%%%%%%%%

\end{picture}

\caption{ The rising water lemma - see (\ref{sup3}).}
\end{figure}

%%%%%%%%%%%%%%%%%%%%%%%%%%%%%%%%%%%

{\bf Remark (7.1)} The supremum  problem in (b) of Theorem 7 is
equivalent to the problem
\begin{equation}
\label{sup3}
\sup_{||z||_1 =\lambda} \ \prod_{i=1}^n (x_i+z_i)  \
\end{equation}
and can be solved
explicitly by a simple construction
that may be called the ``rising water lemma",
by analogy with the ``rising sun lemma"
of real analysis \cite[p. 293, Lemma A]{HLP2}.
Fix $x \in \R_{+}^n $ and consider any $\lambda \geq 0.$
We claim that the supremum (\ref{sup3}) is achieved for
$z =\widetilde{z}$ where $\widetilde{z}$ has the form
\begin{equation}
\label{cl1}
  \widetilde{z}_i = \max(c, x_i) - x_i
\end{equation}
for the (unique) constant $c = c(\lambda)$ such that
$||\widetilde{z}||_1:=\frac{1}{n}\sum_i \widetilde{z}_i = \lambda.$
Before proving the claim, we describe the graphical construction
that justifies calling it the ``rising water lemma" (see Fig. 1):
Without loss of generality, assume $x_1 \leq \dots \leq x_n $
and identify the vector $x \in \R_{+}^n$ with a
step function $f_x: [0,1] \to [0,\infty)$ with steps of width $1/n,$ that is
$f_x = \sum_{i=1}^n \ x_i \mathbf{1}_{[(i-1)/n,i/n)} \ $ (and also  $f_x(1) := x_n$),
whose graph $(\omega, f_x(\omega)), \ 0\leq \omega \leq 1 $ is to be thought of
as the profile of a mountain or water basin. (We assume that the
vertical drops are also part of the graph).
We now imagine a volume of water $\lambda$ (represented in two dimensions as having
total area equal to $\lambda$) poured into the region bounded
below by the graph of $f_x(\omega)$ and bounded on the sides
by the vertical lines $ \omega=0$ and $ \omega=1$. When the water settles,
the constant $c=c(\lambda)$
is the resulting water line (read off the vertical
axis), and $\widetilde{z}_i$ is the depth of the
water in each interval along the graph. (In particular,
$\widetilde{z}_i = 0$ if there is no water above the
$i$th interval).

To prove the claim (\ref{cl1}), observe that for any
vector $z \in \R_{+}^n $ with $||z||_1 = \lambda$ we
have the majorization relation $ x+z \succ x+\widetilde{z}$
(exercise). Hence $\prod_i (x_i+z_i) \leq \prod_i (x_i+\widetilde{z}_i),$
since $\log$ is a concave function. (Corollary: The same $\widetilde{z}$
gives the solution to any problem of the form
$\sup_{||z||_1 =\lambda} \ F(x+z)$ where $F$ is
a Schur-concave function.)

{\bf Remark (7.2)} Remark (7.1) leads to an alternative and direct proof that $a'=b'$ in
the proof of Theorem 7 above. That is,
for fixed $x \in \R_{+}^n \ , \lambda > 0 $ we can give a direct proof that
$$ \inf_{  ||h||_0=1}
\bigg(\sup_{  ||z||_1=\lambda}
 \ T(h,z) \bigg) \
=  \ \sup_{  ||z||_1=\lambda}
\bigg( \inf_{  ||h||_0=1}
 \ T(h,z) \bigg)
  $$
  where
  $$ T(h,z):= \frac{1}{n}\sum_{i=1}^n h_i (x_i
 \  +  \ z_i) ,$$
 and where the vectors $h, z \in \R_{+}^n.$
 As is well known, this follows immediately if we can exhibit a pair $\widetilde{h}, \widetilde{z} \in \R_{+}^n$
 with $||\widetilde{h}||_0=1, \ ||\widetilde{z}||_1=\lambda, $ having the ``saddle point" property,
\begin{equation}
\label{sad}
T(\widetilde{h},z) \leq T(\widetilde{h},\widetilde{z}) \leq T(h,\widetilde{z}) \ , \ \ \ \ \ \ \forall
h, z \in \R_{+}^n \ {\rm with} \ ||h||_0=1, \ ||z||_1=\lambda.
\end{equation}
To do this, let $\widetilde{z}$ be the vector defined by (\ref{cl1}) in Remark
(7.1) above, and define $\widetilde{h}$ by
$$ \widetilde{h}_i = ||x+\widetilde{z}||_0 (x_i+ \widetilde{z}_i)^{-1}. $$
Since $\lambda > 0,$ it follows that $x_i+ \widetilde{z}_i >0$ for all $i,$
so that $\widetilde{h}$ is well defined. It is easy to check that $||\widetilde{h}||_0=1$
and $ T(\widetilde{h}, \widetilde{z}) = ||x+\widetilde{z}||_0.$
For any $h \in \R_{+}^n$ with $||h||_0=1$ we have
by the arithmetic-geometric mean inequality
$$ T(h, \widetilde{z}) \geq ||h \cdot(x+\widetilde{z})||_0 = ||h||_0||x+\widetilde{z}||_0
= ||x+\widetilde{z}||_0 = T(\widetilde{h}, \widetilde{z}),$$
where $v \cdot w$ denotes the coordinate-wise product of vectors;
$(v \cdot w)_i \ := v_iw_i.$ As for the other inequality
in (\ref{sad}), first observe that
$$ \frac{1}{n}\sum_{i=1}^n \widetilde{h}_i \widetilde{z}_i = \lambda||\widetilde{h}||_\infty \ . $$
(This is because the minimum of $(x_i+ \widetilde{z}_i)$ is the constant $c$
in (\ref{cl1}), and $\widetilde{z}_i$ lives on those $i$ where this minimum
is achieved, that is, where $(x_i+ \widetilde{z}_i)=c.$) Therefore,
for any $z \in \R_{+}^n$ with $||z||_1=\lambda,$
$$T(\widetilde{h}, z) = \frac{1}{n}\sum_{i=1}^n \widetilde{h}_i (x_i
 \  +  \ z_i) = \frac{1}{n}\sum_{i=1}^n \widetilde{h}_i x_i
 +\frac{1}{n}\sum_{i=1}^n \widetilde{h}_i z_i$$
 $$\leq \frac{1}{n}\sum_{i=1}^n \widetilde{h}_i x_i + \lambda||\widetilde{h}||_\infty
 = \frac{1}{n}\sum_{i=1}^n \widetilde{h}_i x_i + \frac{1}{n}\sum_{i=1}^n \widetilde{h}_i \widetilde{z}_i
 =T(\widetilde{h}, \widetilde{z}).$$
 {\bf Remark (7.3)} Let $A$ be an $n \times n$ nonegative (semidefinite) Hermitian matrix and
 let $x_1 \geq \dots \geq x_n$ be its eigenvalues. Then
\begin{equation}
\label{supdet}
 \sup_{||z||_1 \leq 1} \ \prod_{i=1}^n (z_i+t x_i) \
 =  \ \sup_{tr(Z) \leq 1} \det (Z+tA), \ \ \  \forall t > 0, \
\end{equation}
where $Z$ runs over all $n \times n$ Hermitian $Z \geq 0$ with trace$(Z) \leq 1$.
To prove this, we invoke the result that adding a nonegative matrix to another
one has the effect of ``increasing" each of its eigenvalues \cite[Ch.7, Theorem 3]{Bell}. Thus, if
the eigenvalues of $(Z+tA)$ are listed as $w_1 \geq \dots \geq w_n$, then
we have $w_i \geq tx_i$ for all $i$. Hence $w_i = y_i + tx_i$ for some $y_i \geq 0$
with $\sum y_i = tr(Z)$. This shows that $\sup_{tr(Z) \leq 1} \det (Z+tA) \leq
\sup_{||z||_1 \leq 1} \ \prod_{i=1}^n (z_i+t x_i)$. The reverse inequality
is obvious.
 \bigskip

 \noindent  {\bf 5.4.
The functional $  \sup_{||z||_1 =1} \ \prod_i (z_i+ t x_i) $ and the polynomials
$F_{k,r}(x)$.}
\medskip

We now relate $\prec_\infty$ (and thus $\prec_L$) to the symmetric polynomials
$F_{k,r}(x)$ (\ref{Fkr}). The link will be provided
by the expression $\prod_i (z_i + tx_i)$ in (b) of Theorem 7, just
as $\prod_i (x_i+\lambda)$ was the link between $\prec_1$ and the
elementary symmetric polynomials (see Theorem 6).
\begin{lemma} Let $n, r \geq 1$ be integers.
Let $T_n = \{ z \in \R_{+}^n \ | \ \sum_{i=1}^n z_i \leq 1 \}.$
There exist constants $C_{n,k,r} > 0$ such that
\begin{equation}
\label{lem8}
\int_{T_n} \ \bigg( \prod_{i=1}^n ( z_i + tx_i ) \bigg)^r \
dz_1 \dots dz_n \
= \ \sum_{k=0}^{nr}C_{n,k,r} F_{k,r}(x)t^k
\end{equation}
for all $t \in \R, \ x \in \R^n .$
\end{lemma}

\begin{proof}
It is well known that
\begin{equation}
\label{beta}
\int_{T_n} \ \left(\prod_{i=1}^n  \frac{z_i^{a_i}}{ a_i!}
\right) \
dz_1 \dots dz_n
\ = \ \frac{1}{(n+ \sum_{i=1}^n a_i)!}  \ \ ,
\end{equation}
for any integers $a_i \geq 0$  \cite[p. 621, eq. 4.635.3]{GR}.
(This can be proved by induction on $n$. Alternatively, one may note
that the Laplace transform of $t^m/m!$ is $1/s^{m+1}$ and use the
convolution theorem.)
Next, by the binomial theorem,\\
$$ \bigg( \prod_{i=1}^n ( z_i + tx_i ) \bigg)^r
 \ = \  \prod_{i=1}^n ( z_i + tx_i )^r
 \ = \ \prod_{i=1}^n \left(
{\displaystyle \sum_{a_i+k_i =r} }
 r! \frac{z_i^{a_i}}{a_i!}
  \frac{x_i^{k_i}}{k_i!}t^{k_i}   \right) \ =: p(z,x,t). $$
Thus $p(z,x,t)$ is of degree $nr$ in $t$ and
the coefficient $c(z,x,k)$ of $t^k$  in $p(z,x,t)$
for a fixed $k$ is
$$
c(z,x,k) \ = \
\sum_{\sum_i k_i = k; \ k_i \leq r; \ a_i = r-k_i }
\ \ \ \ \prod_{i=1}^n
\left( r! \frac{z_i^{a_i}}{a_i!} \frac{x_i^{k_i}}{k_i!}\right).
$$
Note that the conditions under the summation give
 $\sum_i a_i = \sum_i (r-k_i) = nr-\sum_i k_i = nr-k.$
Hence, integrating over $T_n$ and using (\ref{beta}) gives
$$
\int_{T_n} \ c(z,x,k) \
dz_1 \dots dz_n
\ = \
\sum_{\sum_i k_i = k; \ k_i \leq r; \ a_i = r-k_i }
\ \ \ \
\int_{T_n} \
\prod_{i=1}^n
\left( r! \frac{z_i^{a_i}}{a_i!} \frac{x_i^{k_i}}{k_i!}\right)
\
dz_1 \dots dz_n
$$
$$
=\
\sum_{\sum_i k_i = k; \ k_i \leq r }
\ (r!)^n \frac{1}{(n+(nr-k))!}
\ \prod_{i=1}^n
 \frac{x_i^{k_i}}{k_i!} \ \
 = \ \ \frac{(r!)^n }{(nr+n-k)!} \ F_{k,r}(x) \ ,
$$
which is the desired result with the constants $C_{n,k,r}= \frac{(r!)^n }{(nr+n-k)!}.$
\end{proof}

{\bf Remark.} There is an integral identity similar to (\ref{lem8})
but having the generating function (\ref{gen1})
on the right side as follows:
%%%%%%%%%%%%%%
\begin{equation}
\label{gen2}
 \int_{\R_{+}^n} \
  \prod_{i=1}^n \frac{(z_i + tx_i)^r}{r!}
 \ \ e^{-(z_1 + \dots + z_n)}
\ dz_1 \dots dz_n \  =  \ \prod_{i=1}^n \ P_r(x_it) \ =:
\sum_{k=0}^{nr} F_{k,r}(x)t^k \ :=: \ f_r(x,t).
\end{equation}
This can be seen immediately by factoring the integral in (\ref{gen2}) and
using the fact that $\int_0^\infty (u+s)^r \ e^{-u}du/r!
= 1 + s + \dots + \frac{s^r}{r!} =: P_r(s)$ (the ``incomplete gamma integral").
We have however not found any
uses for the identity (\ref{gen2}) in the present paper.
%%%%%%%%%%%%%%%%%%%%%

For the next result, recall the definitions of $x \prec_F y$ (\ref{F}),
$x\prec_\infty y$ (\ref{dhy}) , and $x\prec_L y$ (\ref{L}).

\begin{theorem} Let $x,y \in \R_{+}^n.$ If $x \prec_F y$ then $x\prec_L y$
{\rm(}equivalently, $x\prec_\infty y$, by Theorem 7{\rm)}.
\end{theorem}
\begin{proof}
If $x \prec_F y$ then by definition $ F_{k,r}(x) \leq F_{k,r}(y)$ for all
$k,r.$ Now fix any $t>0$ and let $r \geq 1 $ be an integer. Lemma 8 shows that
$$|| \prod_{i=1}^n ( z_i + tx_i ) \ ||_{L^r(dz)}
\ \leq \ || \prod_{i=1}^n ( z_i + ty_i ) \ ||_{L^r(dz)}
$$
where the $L^r$-norm is taken with respect to $dz:=dz_1 \dots dz_n \ $, over the set $T_n=
\{\sum z_i \leq 1\}\cap \R_{+}^n$
used in Lemma 8.
Letting $r\to \infty$ gives the same inequality for
the $L^\infty$-norm. This gives (b) of Theorem 7, since the
suprema over $z\in T_n$ are achieved on
$\sum z_i = 1.$
\end{proof}

\noindent {\Large {\bf 6. Some results on the conjecture
``$\ x \prec_L y \ \Rightarrow \ x \prec_F y \ $".} }

The next easy lemma is one possible approach to proving statements
of the type ``$\ x \prec_L y \ \Rightarrow \ x \prec_F y \ $". (Although
it shifts the problem to the existence of a
certain path $\gamma$.) Given a family $\Phi:= \{\Phi_\lambda\}_{\lambda \in \Lambda}$
of real-valued functions on $\R^n$ differentiable at a point $p \in \R^n$, define
\begin{equation}
\label{cone}
C_{\Phi}(p) :=\{ c_1\nabla \Phi_{\lambda_1}(p) + \dots + c_k\nabla \Phi_{\lambda_k}(p)
\ | \ k \in \N, c_i \geq 0, \lambda_i \in \Lambda \}.
\end{equation}
The closure $\overline{C_{\Phi}(p)}$ will be called the
``positive cone" generated by the
gradients $\{\nabla \Phi_\lambda(p)\}_{\lambda \in \Lambda}\ $.

\begin{lemma}
Let $\Lambda$ be any index set, let $\Phi:= \{\Phi_\lambda\}_{\lambda \in \Lambda}$ be
a family of $C^1$ real-valued functions on $\R^n$,
 and let $x,y \in \R^n$
be two points such that $ \forall \lambda \in \Lambda,
\ \Phi_\lambda(x) \leq \Phi_\lambda(y).$ Suppose there exists
a $C^1$ path $\gamma : [0,1] \to \R^n$ such that

{\bf (a)} $\ \gamma(0) = x, \ \gamma(1) = y $, \ \ and  \ \
{\bf (b)} $\ 0\leq s \leq t \leq 1 \Rightarrow \forall \lambda \in \Lambda,
\ \Phi_\lambda(\gamma(s)) \leq \Phi_\lambda(\gamma(t))  $.

\noindent
Let $F$ be any $C^1$ real-valued function on a neighbourhood of $\gamma([0,1])$
such that

{\bf (c)} $\nabla F(p) \in \overline{C_{\Phi}(p)}$ for all $p \in \gamma ([0,1])$
 {\rm [see (\ref{cone})]}.

\noindent Then $F(x) \leq F(y).$
\end{lemma}

\begin{proof} Given the hypotheses, it suffices to verify that
$\forall t \in [0,1], \ \frac{d}{dt} F(\gamma(t)) \geq 0 .$
By (b) we have $\frac{d}{dt} \Phi_{\lambda}(\gamma(t)) \geq 0 $
for all $\lambda$ and all $t.$ Thus
$\nabla \Phi_{\lambda}(\gamma(t))\cdot \gamma'(t) \geq 0 ,$
by the chain rule.
Hence, for a fixed $t \in [0,1]$ we have $v\cdot \gamma'(t) \geq 0$
for all $v \in C_{\Phi}(\gamma(t))$ and thus also
for all $v$ in the closure of $ C_{\Phi}(\gamma(t)).$ So,
$ \frac{d}{dt} F(\gamma(t)) = \nabla F(\gamma(t))\cdot \gamma'(t)\geq 0 .$
\end{proof}

{\bf Remark (10.1)} Conversely, the property $\nabla F(p) \in \overline{C_{\Phi}(p)}$
is necessary in the following sense.
If $F$ and $\{\Phi_\lambda\}_{\lambda \in \Lambda}$ are
sufficiently regular on some open set $U$
(e.g. if the gradients $\nabla \Phi_{\lambda}$ are
nonzero and the $\nabla \Phi_{\lambda}/||\nabla \Phi_{\lambda}||$ are
equicontinuous on $U$), and
the implication $ \forall \lambda \in \Lambda,
\ \Phi_\lambda(x) \leq \Phi_\lambda(y) \Rightarrow F(x) \leq F(y)$
is known to hold for all $x,y \in U,$ then for all $p \in U, \ $
$\nabla F(p) \in \overline{C_{\Phi}(p)}$. Proof: If
this does not hold at some $p \in U$,
then we can find a vector $h \in \R^n$ and a fixed
$\delta > 0$ such that
$\nabla F(p)\cdot h < 0 $ and $ v\cdot h \geq \delta $
for all unit vectors $v \in
\overline{C_{\Phi}(p)}$, by a basic separation theorem of convex
analysis. It follows that there is a point $q=p+\epsilon h \in U$ near $p$
such that $F(p) > F(q)$ while $\forall \lambda \in \Lambda,
\ \Phi_\lambda(p) \leq \Phi_\lambda(q).$

{\bf Remark (10.2)} In the special case when $\Lambda = \{1,\dots , n\}$
and $\Phi := (\Phi_1, \dots, \Phi_n)$ is a $C^1$ diffeomorphism of an open set
$U$ onto a convex set in $\R^n$,
Marshall et al. \cite[Corollary 8]{MWW} shows that
the condition $\forall p \in U, \nabla F(p) \in \overline{C_{\Phi}(p)}$
is necessary and sufficient for $F$ to have the order-preserving
property $x,y \in U,  \forall \lambda \in \Lambda,
\ \Phi_\lambda(x) \leq \Phi_\lambda(y) \Rightarrow F(x) \leq F(y)$.
Under these hypotheses, a path $\gamma$ (in $U$) satisfying (a) and (b) of Lemma 10
exists automatically; take $\gamma(t):= \Phi^{-1}( (1-t)\Phi(x) + t\Phi(y) )$.

We now consider the example $\Lambda:= (0, \infty)$ and $\Phi_\lambda := \Psi_\lambda$
given by (\ref{psib}), extended if necessary to all $x\in \R^n$, so that
$ \Psi_\lambda(x) := \sum_{i=1}^n \psi_\lambda(x_i),
\ x \in \R^n, \ \lambda > 0,$ where $\psi_\lambda : \R \to \R $ is defined by
$$ \psi_\lambda(s)= \min(s,\lambda) + \lambda \log_+(s/\lambda) ,
 \ \ \ \ \ \ \  s \in \R, \ \lambda > 0,
$$
and by convention $\log_+(s):= 0 $ for all $s \in (-\infty,1]$
and $\log_+(s):= \log s $ for all $s \geq 1.$
Clearly, the $ \psi_\lambda$'s are $C^1$
in $s$ on all of $\R,$ and thus the $\Psi_\lambda$ are $C^1$ on $\R^n.$
We have
$$ \frac{\partial}{\partial s} \psi_\lambda(s)
= \left\{
\begin{array}{l}
 1 , \ \ \ \ \ {\rm if } \ s \leq \lambda, \\
\lambda/s, \ \ \ \ {\rm if }  \ s \geq \lambda .\\
\end{array}
\right.
$$
Thus, if  $x \in \mathcal{D}_n :=
 \{x\in \R^n \ | \ x_1 \geq \dots \geq x_n \geq 0\}$
and $\lambda > 0$ is such that
$x_1 \geq \dots \geq x_k \geq \lambda \geq
x_{k+1} \geq \dots \geq x_n \geq 0,$ then
$$\nabla \Psi_{\lambda}(x) = ( \frac{\lambda}{x_1}, \dots, \frac{\lambda}{x_k},
1,1,\dots, 1).$$

Our goal is to derive a simple criterion for a vector $B$ to belong
to the positive cone $\overline{C_\Psi(x)}$
of such gradients (see (\ref{cone})) at a fixed $x \in \mathcal{D}_n.$ For convenience,
consider the case $ x_1 > x_2 > \dots > x_n > 0.$ Multiplying each gradient
by $\frac{1}{\lambda}$, we may consider
our cone $\overline{C_\Psi(x)}$ to be generated
by the (uncountable) collection of vectors
$$ ( \frac{1}{x_1}, \dots, \frac{1}{x_k},
\frac{1}{\lambda},\frac{1}{\lambda},\dots, \frac{1}{\lambda})
  \ \ \ \ \ \ \ \ \  1\leq k \leq n-1,\ \ \  \lambda \in
  [x_k, x_{k+1}] \ .$$
But it suffices to keep just the $n$ vectors
$$ R_k := ( \frac{1}{x_1}, \dots, \frac{1}{x_k},
\frac{1}{x_k},\frac{1}{x_k},\dots, \frac{1}{x_k}) \ \ \ \ \ \ \ \ \  1\leq k \leq n,
$$
since it is clear that the other vectors are on the line segments
between the $R_k$ and $R_{k+1}.$ This shows that
$C_\Psi(x)= \{ c_1R_1 + \dots + c_nR_n  \ | \ c_i \geq 0 \}=\overline{C_\Psi(x)}$.
Let $R$ be the $n \times n$ matrix whose columns are the transposes
of the $R_k, \ k = 1,\dots , n$. (Clearly $R=R^T$.)
The linear system $RC=B^T$ (where
$C:=(c_1,\dots, c_n)^T,\ B:=(B_1, \dots, B_n)$)
can be row-reduced to diagonal form in an obvious
way, which we now illustrate for $n=4$ without loss of generality.
Put $A_k := \frac{1}{x_k},$
so that $ 0 < A_1 < A_2 < A_3 < A_4.$
The augmented matrix of the system $RC=B^T$ is then
$$
\left[
\begin{array}{cccc}
A_1 & A_1 & A_1 & A_1 \\
A_1 & A_2 & A_2 & A_2 \\
A_1 & A_2 & A_3 & A_3 \\
A_1 & A_2 & A_3 & A_4 \\
\end{array}
\left|
\begin{array}{c}
B_1\\
B_2\\
B_3\\
B_4\\
\end{array}\right] \right. \ .
$$
Subtracting row 3 from row 4, then row 2 from row 3, and finally
row 1 from row 2 gives
$$
\left[
\begin{array}{cccc}
a_1 & a_1 & a_1 & a_1 \\
0 & a_2 & a_2 & a_2 \\
0 & 0 & a_3 & a_3 \\
0 & 0 & 0 & a_4 \\
\end{array}
\left|
\begin{array}{c}
b_1\\
b_2\\
b_3\\
b_4\\
\end{array}\right] \right.
$$
where $a_k := A_k-A_{k-1}, \dots, \ a_1 := A_1 $ and similarly
$b_k := B_k-B_{k-1}, \dots, \ b_1 := B_1 .$ Now subtracting the
appropriate multiples of row 2 from row 1, row 3 from row 2,
and row 4 from row 3 gives
$$
\left[
\begin{array}{cccc}
a_1 & 0 & 0 & 0 \\
0 & a_2 & 0 & 0 \\
0 & 0 & a_3 & 0 \\
0 & 0 & 0 & a_4 \\
\end{array}
\left|
\begin{array}{c}
b_1-b_2(a_1/a_2)\\
b_2-b_3(a_2/a_3)\\
b_3-b_4(a_3/a_4)\\
b_4\\
\end{array}\right] \right. \ \ .
$$
This shows that a vector $B=(B_1, \dots, B_n)$ belongs
to $C_\Psi(x)$ if and only if $b_k-b_{k+1}(a_k/a_{k+1}) \geq 0,
\ k = 1, \dots, n-1$ and $b_n \geq 0,$ in the above notation.
Note that each of these inequalities involves at most 3 consecutive
 $A_k$ and $B_k.$ Hence,
  $B \in C_\Psi(x)$  if and only if each 3-vector
 $(B_k, B_{k+1}, B_{k+2})$ is in the projection of $C_\Psi(x)$ on these
 same 3 coordinates, which is just the positive cone in
 $\R^3$ generated by the three vectors $(A_k,A_k,A_k),(A_k,A_{k+1},A_{k+1}),
(A_k,A_{k+1},A_{k+2}).$
We may also observe that, if $B \in C_\Psi(x),$ then for {\it any}
3 coordinates $1\leq i < j < k \leq n$ (not necessarily consecutive),
the same situation holds - that is, $(B_i, B_j, B_k)$ must be in
the positive cone generated by $(A_i,A_i,A_i),(A_i,A_j,A_j),
(A_i,A_j,A_k).$ To see this, note that the projections of all
$R_m$ on the given coordinates $i,j,k$ are
on the line segments joining the latter three vectors.
(An analogous statement holds for any $g$
coordinates $1\leq i_1 < \dots < i_g \leq n$, not just $g=3$.)
Summarizing, we have:
\begin{lemma} Let $x \in \R^n_+ $ with $ x_1 > x_2 > \dots > x_n > 0$
and let $B=(B_1, \dots, B_n) \in \R^n.$ Then $B$ is in $\overline{C_\Psi(x)}$
{\rm [see (\ref{psib}) and (\ref{cone})]} if and only if

{\rm (a) } $0\leq B_1 \leq \dots \leq B_n $,

{\rm (b) } $B_i\frac{1}{x_j}-B_j\frac{1}{x_i} \geq 0$
whenever $1\leq i < j \leq n,$ and

{\rm (c) } $ (B_j - B_i)(\frac{1}{x_k} - \frac{1}{x_j})-
(B_k - B_j)(\frac{1}{x_j} - \frac{1}{x_i}) \geq 0$
whenever $1\leq i < j < k \leq n.$
\end{lemma}
Conditions ${\rm (a)},{\rm (b)},{\rm (c)}$ may be stated more
concisely by saying that the matrix
\begin{equation}
\label{tp3a}
\left[
\begin{array}{ccc}
1 & 1 & 1 \\
B_i & B_j & B_k  \\
  \frac{1}{x_i} & \frac{1}{x_j} & \frac{1}{x_k}
\end{array}
\right]
\end{equation}
is totally positive. The fact that
it is possible to state all three of the conditions (a)-(c)
in terms of determinants in this manner is clear from the
discussion preceding the lemma. (The linear
systems $RC=B^T$ (of various sizes) could have been solved for $C$
by Cramer's rule instead of row operations, and the signs
of the $c_j$ in $C$ are determined by the numerators in Cramer's formula,
since $\det R > 0$ as already shown.)
We can further
remark that the only property of the sequence $A_k = \frac{1}{x_k}$
that was used was that it is strictly positive and strictly increasing.
In this generality, Marshall et al. \cite[Example 3]{MWW} notes that conditions (a)-(c)
are just another way of saying that the sequence $\{B_k\}$ is of the
form $B_k = \varphi(A_k)$ for some nondecreasing concave function
$\varphi$ on $[0,\infty)$ with $\varphi(0)=0$.

The extension of Lemma 11 to points $x \in \mathcal{D}_n$
is straightforward: If $x=0$ then the cone $\overline{C_\Psi(x)}$ is generated
by a single vector, $(1,1,\dots,1).$ If
$x_1 \geq \dots \geq x_m > 0=x_{m+1} = \dots = x_n$
for some fixed $1 \leq m < n$ then $\overline{C_\Psi(x)}$ is generated by
the above
$R_k $ for just $k=1, \dots ,m$ together with the
vector $R_\infty :=( 0, \dots, 0,
1,1,\dots, 1) $ having $0$ in the first $m$ entries, followed by $1$'s.
Thus
$\overline{C_\Psi(x)} = \{ c_1R_1 + \dots + c_mR_m +
c_{m+1}R_\infty  \ | \ c_i \geq 0 \}.$
We deduce that:
\begin{lemma} If $0 \neq x \in \mathcal{D}_n $ and
$x_1 \geq \dots \geq x_m > 0=x_{m+1} = \dots = x_n$, then
 a vector $B\in \R^n$ belongs to $\overline{C_\Psi(x)}$
{\rm [see (\ref{psib}) and (\ref{cone})]} if and only if

$B_m \leq B_{m+1} = \dots =B_n$ and {\rm (\ref{tp3a})}
is totally positive whenever $1\leq i \leq j \leq k \leq m$.
\end{lemma}
\begin{lemma}
Let $n \geq 1$ and let $F=F_{k,r}$ for some $k, r \geq 1.$
Then $\forall i \ \frac{\partial F}{\partial x_i} \geq 0$ on $\R_{+}^n$, and
\begin{equation}
\label{tp2a}
\left|
\begin{array}{cc}
 1 & 1 \\
  \frac{\partial F}{\partial x_i}(x)
 & \frac{\partial F}{\partial x_j}(x)
\end{array}
\right| \ \geq \ 0 \
\end{equation}
for all $x \in \R_{+}^n$ and $i,j$ such that $x_i \geq x_j . $ Furthermore,
\begin{equation}
\label{tp2b}
\left|
\begin{array}{cc}
  \frac{\partial F}{\partial x_i}(x)
 & \frac{\partial F}{\partial x_j}(x)\\
  \frac{1}{x_i} & \frac{1}{x_j}
\end{array}
\right| \ \geq \ 0 \ ,
\end{equation}
and
\begin{equation}
\label{tp3}
\left|
\begin{array}{ccc}
1 & 1 & 1 \\
  \frac{\partial F}{\partial x_i}(x)
 & \frac{\partial F}{\partial x_j}(x)
 & \frac{\partial F}{\partial x_k}(x)  \\
  \frac{1}{x_i} & \frac{1}{x_j} & \frac{1}{x_k}
\end{array}
\right| \ \geq \ 0
\end{equation}\\
\noindent
whenever $1 \leq i,j,k \leq n $
and $x \in \R_{+}^n$ is such that $x_i \geq x_j \geq x_k > 0 $.
{\rm [For such $x,$ conditions {\rm (\ref{tp2a}),(\ref{tp2b}),(\ref{tp3})} (together
with the remarks that $ \frac{\partial F}{\partial x_i} \geq 0 \ \forall i $
and $ \frac{1}{x_2} - \frac{1}{x_1}
 \geq \ 0$  when $x_1 \geq x_2 > 0$)
 are equivalent to the statement that the $3 \times 3$ matrix in
{\rm (\ref{tp3})} is
totally positive. If $x_i > x_j > x_k >0 ,$
condition {\rm (\ref{tp3})} can be written as
$x_k(\frac{\partial F}{\partial x_k}-\frac{\partial F}{\partial x_j})
/(x_j-x_k) \leq
x_i(\frac{\partial F}{\partial x_j}-\frac{\partial F}{\partial x_i})
/(x_i-x_j)$ .}
{\rm ]}
\end{lemma}

\begin{proof}
\noindent It is clear that $\forall i \ \frac{\partial F}{\partial x_i} \geq 0$ on $\R_{+}^n$.
Condition (\ref{tp2a}) is the Schur-concavity condition,
which was already discussed for the more general class of polynomials (\ref{HS}).
 We proceed to condition (\ref{tp2b}). This actually
holds for any symmetric polynomial $F$ with positive coefficients.
In fact, by linearity, it will suffice to prove (\ref{tp2b}) when
$n=2, \ (x_i,x_j)=(x_1,x_2):=:(x,y)$ with $x \geq y \geq 0, $ and
$$
F:= F(x,y) := x^ay^b + x^by^a
$$
for some integers $a,b \geq 0.$ We get
$$
\left|
\begin{array}{cc}
 x \frac{\partial F}{\partial x}
 & y\frac{\partial F}{\partial y}\\
  1 & 1
\end{array}
\right|
= (a-b)(x^ay^b - x^by^a) \geq 0.
$$
The latter inequality is easy to verify, and is a special
case of the fact that the kernel $K(s,p)=s^p$ is totally positive
on $\R_+ \times \R$ \cite[Example 18.A.6.a]{MO}.
We now prove (\ref{tp3}) for $F=F_{k,r}.$
By symmetry and linearity, it suffices to prove (\ref{tp3}) when
$n=3, \ (x_i,x_j,x_k)=(x_1,x_2,x_3):=:(x,y,z)$ with $x \geq y \geq z > 0. $
Recalling the generating function (\ref{gen1}) $f_r(x,y,z,t)
= \sum_k F_{k,r}(x,y,z)t^k$ , it is clear that the determinant in (\ref{tp3})
is the coefficient of $t^k$ in
$$
\frac{1}{xyz}
\left|
\begin{array}{ccc}
x & y & z \\
 x \frac{\partial f_r}{\partial x}
 & y\frac{\partial f_r}{\partial y}
 & z\frac{\partial f_r}{\partial z}  \\
  1 & 1& 1
\end{array}
\right| \ .
$$
By definition, $ f_r(x,y,z,t) = P_r(xt)P_r(yt)P_r(zt)$ where $P_r$
is the polynomial in one variable given by
$$P_r(s) = \sum_{m=0}^r \ \frac{s^m}{m!} \ .$$
The $P_r$ have the nice properties
$$
P_r(s) = P_{r-1}(s) + \frac{s^r}{r!} \ , \ \  {\rm and}
\ \ \ \frac{d}{ds}P_r(s) = P_{r-1}(s) \ , \ \ \ \ \ \ \ \ (r \geq 0)
$$
where $P_{-1} = 0$ for convenience.
It follows that
$$
\left|
\begin{array}{ccc}
x & y & z \\
 x \frac{\partial f_r}{\partial x}
 & y\frac{\partial f_r}{\partial y}
 & z\frac{\partial f_r}{\partial z}  \\
  1 & 1& 1
\end{array}
\right|
=
\left|
\begin{array}{ccc}
xP_r(xt) & yP_r(yt) & zP_r(zt) \\
 x \frac{\partial }{\partial x}P_r(xt)
 & y\frac{\partial }{\partial y}P_r(yt)
 & z\frac{\partial }{\partial z}P_r(zt)  \\
  P_r(xt) & P_r(yt)& P_r(zt)
\end{array}
\right|
$$
$$
=
\left|
\begin{array}{ccc}
x(P_{r-1}(xt)+ \frac{(xt)^r}{r!})
& y(P_{r-1}(yt)+ \frac{(yt)^r}{r!})
& z(P_{r-1}(zt)+ \frac{(zt)^r}{r!}) \\
 x P_{r-1}(xt)t
 & y P_{r-1}(yt)t
 & z P_{r-1}(zt)t  \\
  P_r(xt) & P_r(yt)& P_r(zt)
\end{array}
\right|
$$
$$
=
\left|
\begin{array}{ccc}
 \frac{x^{r+1}t^r}{r!}
& \frac{y^{r+1}t^r}{r!}
& \frac{z^{r+1}t^r}{r!} \\
 x P_{r-1}(xt)t
 & y P_{r-1}(yt)t
 & z P_{r-1}(zt)t  \\
  P_r(xt) & P_r(yt)& P_r(zt)
\end{array}
\right|
 \ .
$$

Introducing the row vectors
$\overrightarrow{V_m} := \frac{1}{m!}[x^m \ , \ y^m \ , \ z^m]$
we may write the latter determinant as
$$
\left|
\begin{array}{ccc}
 \frac{x^{r+1}t^r}{r!}
& \frac{y^{r+1}t^r}{r!}
& \frac{z^{r+1}t^r}{r!} \\
 x P_{r-1}(xt)t
 & y P_{r-1}(yt)t
 & z P_{r-1}(zt)t  \\
  P_r(xt) & P_r(yt)& P_r(zt)
\end{array}
\right|
\ = \
\left|
\begin{array}{c}
 (r+1)\overrightarrow{V_{r+1}}t^{r} \\
 \sum_{a=1}^r \ a\overrightarrow{V_a}t^a  \\
 \sum_{b=0}^r \ \overrightarrow{V_b}t^b \\
\end{array}
\right| \ =
$$\\
$$= \
\sum_{a=1}^r
\left|
\begin{array}{c}
 (r+1)\overrightarrow{V_{(r+1)}}t^{r} \\
  a\overrightarrow{V_a}t^a  \\
  \overrightarrow{V_0}t^0 \\
\end{array}
\right| \ +  \bigg( \
\sum_{r\geq a>b \geq 1}
\left|
\begin{array}{c}
 (r+1)\overrightarrow{V_{(r+1)}}t^{r} \\
  a\overrightarrow{V_a}t^a  \\
  \overrightarrow{V_b}t^b \\
\end{array}
\right| \ + \
\sum_{r\geq b>a \geq 1}
\left|
\begin{array}{c}
 (r+1)\overrightarrow{V_{(r+1)}}t^{r} \\
  a\overrightarrow{V_a}t^a  \\
  \overrightarrow{V_b}t^b \\
\end{array}
\right| \ \bigg)
$$ \\
%%%%%%
$$= \
\sum_{a=1}^r
\left|
\begin{array}{c}
 (r+1)\overrightarrow{V_{(r+1)}}t^{r} \\
  a\overrightarrow{V_a}t^a  \\
  \overrightarrow{V_0}t^0 \\
\end{array}
\right| \ +  \bigg( \
\sum_{r\geq a>b \geq 1}
\left|
\begin{array}{c}
 (r+1)\overrightarrow{V_{(r+1)}}t^{r} \\
  a\overrightarrow{V_a}t^a  \\
  \overrightarrow{V_b}t^b \\
\end{array}
\right| \ + \
\sum_{r\geq a>b \geq 1}
\left|
\begin{array}{c}
 (r+1)\overrightarrow{V_{(r+1)}}t^{r} \\
  b\overrightarrow{V_b}t^b  \\
  \overrightarrow{V_a}t^a \\
\end{array}
\right| \ \bigg)
$$ \\
%%%%%%%
%%%%%%
$$=\sum_{a=1}^r
\left|
\begin{array}{c}
 \overrightarrow{V_{(r+1)}} \\
 \overrightarrow{V_a} \\
  \overrightarrow{V_0} \\
\end{array}
\right|(r+1)at^{r+a}
 +
\sum_{r\geq a>b \geq 1}
\left|
\begin{array}{c}
 \overrightarrow{V_{(r+1)}} \\
 \overrightarrow{V_a} \\
  \overrightarrow{V_b} \\
\end{array}
\right|(r+1)(a-b)t^{r+a+b} $$\\
$$ =
\ \sum_{r\geq a>b \geq 0}
\left|
\begin{array}{c}
 \overrightarrow{V_{(r+1)}} \\
 \overrightarrow{V_a} \\
  \overrightarrow{V_b} \\
\end{array}
\right|(r+1)(a-b)t^{r+a+b}
\ .
$$ \\
\noindent
In the latter, the coefficient of each $t^{r+a+b}$
is nonnegative,
since $r+1 > a > b \geq 0$, $x\geq y \geq z \geq 0$,
and since the kernel $K(s,p)=s^p$ is totally positive.
Hence, after collecting the like powers among the $t^{r+a+b},$
the final coefficient of each $t^k$ will be nonnegative also.
[Remark: The determinants of the $\overrightarrow{V_m}$'s
are known as generalized Vandermonde determinants,
if we ignore the scalar factors $\frac{1}{m!}.$]
\end{proof}

\noindent {\bf Exercise. } Consider $F=H_S$ where $S \subset I_k$ is
a Schur-concave index set (see (\ref{HS})). Show that: {\bf (a) }  There exist $n,k$ and $S$
such that $F$ does not satisfy (\ref{tp3}) of Lemma 13.
{\bf (b) } However, if $p_i \leq 2 $ for all $i$ and all $p \in S$, then
$F$ does satisfy (\ref{tp3}).
{\bf (c) } Hence there exists an $F$ of the form $F=H_S$ which is not of the form
$F=F_{k,r}$ but satisfies (\ref{tp3}) and all of
the other conditions  of Lemma 13.

\begin{corollary} Let
$ \Psi_\lambda(x) := \sum_{i=1}^n \min(x_i,\lambda) + \lambda \log_+(x_i/\lambda),
\ x \in \R^n, \ \lambda > 0$ {\rm [where $\log_+(s):= 0 $ for all $s \in (-\infty,1]$
and $\log_+(s):= \log s $ for all $s \geq 1$].} If $x \in \R_{+}^n$ then
\begin{equation}
\label{grad}
\nabla F_{k,r}(x) \ \in \ \overline{C_\Psi(x)}
\end{equation}
for all $k,r \geq 1,$ where
$\overline{C_\Psi(x)}$ is the positive cone generated by
$\{\nabla \Psi_{\lambda}(x)\}_{ \lambda > 0 }$ {\rm [see (\ref{cone})]}.
\end{corollary}

\begin{proof} Since both the $\Psi_\lambda(x)$ and the $F_{k,r}(x)$
are symmetric functions, it suffices to consider $x$ such that
$x_1 \geq x_2 \geq \dots \geq x_n \geq 0.$ Let
$B:= \nabla F_{k,r}(x)$.
By the symmetry of $F_{k,r}$
we have $B_i=B_j$ whenever $x_i=x_j.$ This together with Lemma 13
shows that $B$ satisfies all the conditions
of Lemma 12, thus giving (\ref{grad}).
\end{proof}

Corollary 14 and Lemma 10 yield
the following partial converse of Theorem 9:
\begin{theorem}
 Let $x,y \in \R_{+}^n$
be two points such that $x \prec_L y.$
 Suppose there exists
a $C^1$ path $\gamma : [0,1] \to \R_{+}^n$ with
$\gamma(0) = x, \ \gamma(1) = y $ and such that
$0\leq s \leq t \leq 1 \Rightarrow  \gamma(s)\prec_L \gamma(t).$
Then $x \prec_F y.$
\end{theorem}
\begin{proof} By hypothesis, we have a path
$\gamma$ with (a) and (b) of Lemma 10 for $\Phi_\lambda := \Psi_\lambda$
(see (\ref{psib})). Part (c) of Lemma 10
holds for each $F=F_{k,r}$  by Corollary 14, since
the path $\gamma$ is in $ \R_{+}^n.$ Thus Lemma 10 implies
that $F_{k,r}(x) \leq F_{k,r}(y) $ for all $k,r \geq 1.$
\end{proof}

It seems intuitively plausible that the path $\gamma$ in the hypothesis
of Theorem 15 does exist when both $x$ and $y$
are in $\mathcal{D}_n :=
 \{x\in \R^n \ | \ x_1 \geq \dots \geq x_n \geq 0\}$. (If so, then Theorem 15 would be a
complete converse of Theorem 9, thus establishing the equivalence
of the relations $\prec_L$ and $\prec_F$.)
We propose the following specific conjecture on the existence of $\gamma$,
in the more general context of Lemma 10.
\begin{conjecture}
Let $\Lambda$ be any index set, totally ordered by some order $\ll$, and consider
any real-valued functions
$\{\Phi_\lambda\}_{\lambda \in \Lambda}$ on $\mathcal{D}_n$ which
are $C^1$ on $\mathcal{D}_n$  for each fixed $\lambda.$
Suppose that the gradient ``matrix"
$\{\nabla \Phi_\lambda\}_{\lambda \in \Lambda}$ is totally positive on $\mathcal{D}_n$
{\rm (}i.e. $\det [\partial \Phi_{\lambda_i} / \partial x_{j_k}] \geq 0$
on $\mathcal{D}_n$ whenever $\{\lambda_1 \ll \dots \ll \lambda_r \} \subset \Lambda$,
$1\leq j_1 \leq \dots \leq j_r \leq n$,
and $1 \leq r \leq n${\rm ).}
Let $x,y \in \mathcal{D}_n$ be such that $\Phi_\lambda(x) \leq
\Phi_\lambda(y) \ \forall \lambda \in \Lambda .$ Then there is a
$C^1$ path $\gamma : [0,1] \to \mathcal{D}_n$ such that
{\bf (a)} $\gamma(0) = x, \ \gamma(1) = y $, and
{\bf (b)} $0\leq s \leq t \leq 1 \Rightarrow \forall \lambda \in \Lambda,
\ \Phi_\lambda(\gamma(s)) \leq \Phi_\lambda(\gamma(t)) . $
\end{conjecture}

The following examples are possible cases of Conjecture 16. The existence of
the path $\gamma$ appears to be a non-trivial problem in all of them except
1b and 2. The methods of \cite{Ger} may be applicable to this problem.
In Example 3, both the existence of $\gamma$ and the total positivity
hypothesis of Conjecture 16 seem to be non-trivial to verify.

{\bf Example 1.} $\Phi_{\lambda}(x) = \sum_{i=1}^n \phi_\lambda(x_i)$ for some
family of functions $\{\phi_\lambda\}_{\lambda \in \Lambda}$.
Here the hypothesis of total positivity of the gradient matrix
$\{\nabla \Phi_\lambda\}_{\lambda \in \Lambda}$
specializes to requiring that $\partial \phi_\lambda(s)/{\partial s}=:
K(s,\lambda)$ is a totally positive kernel on $[0,\infty) \times \Lambda$
(with the decreasing order on $[0,\infty)$ and the given order $\ll$ on
$\Lambda$).
Specific cases are:

{\bf (1a.)} $ \Lambda = (0,\infty)$ (ordered by
$\lambda_1 \ll \lambda_2 \Leftrightarrow \lambda_1 \geq \lambda_2 $)
with $\phi_\lambda(s)=\psi_\lambda(s)
:= \min(s,\lambda) + \lambda \log_+(s/\lambda)$ as in (\ref{psib}).
(This is the case of Conjecture 16 needed in Theorem 15.)
The kernel is $K(s,\lambda) =
\partial \psi_\lambda(s)/{\partial s} = \min(s,\lambda)/s$. Its
total positivity on $[0, \infty) \times (0, \infty)$ is easy to verify
-- see the proof of Lemma 12  or \cite[Theorem 18.A.4, Examples 18.A.7, 18.A.7a]{MO}.

{\bf (1b.)} $ \Lambda = (0,\infty)$ with $\phi_\lambda(s)= \min(s,\lambda)$.
Here the $C^1$ requirement is not quite satisfied. Nevertheless, the
kernel $K(s,\lambda) =
\partial \phi_\lambda(s)/{\partial s} = {\bf 1}_{\{s<\lambda\}}$ is totally
positive, and the relation induced by the $\Phi_{\lambda}$ is $\prec^w$ (see \S1).
Moreover, the existence of the path $\gamma$ is known in this example; just
take the straight line segment from $x$ to $y$. (This is easier to
verify in Example 2 below, which induces the same relation).

{\bf (1c.)} The ``power sums"
$\Phi_{\lambda}(x):= \sum_i x_i^\lambda,$
given by the choice $\phi_\lambda(s) = s^\lambda, \ \lambda \in
\Lambda$ where $\Lambda \subset \R $.
The kernel $K(s,\lambda) =\partial \phi_\lambda(s)/{\partial s}= \lambda s^{\lambda-1}$
is totally positive on $[0,\infty) \times \R$ \cite[Example 18.A.6.a]{MO}.

{\bf Example 2.} $\Lambda = \{1,\dots,n\}$ with $\Phi_\lambda(x) =
x_\lambda+x_{\lambda+1} + \dots + x_n$. The relation induced by these
$\Phi_\lambda$ is again $\prec^w$, as in (1b) above, but this is formally a distinct
case of Conjecture 16. We can clearly use a line segment for the path $\gamma$.
(In fact, the image $\Phi(\mathcal{D}_n)$ is convex, so
the results of Marshall et al.
\cite{MWW} suffice for a complete discussion of order
preservation - see Remark (10.2) after Lemma 10.)

{\bf Example 3.} $\Lambda = \{1,\dots,n\}$ with
 the elementary
symmetric polynomials $\Phi_\lambda(x):= E_\lambda(x)$ (see \S2.3).
(Note that $\Phi(\mathcal{D}_n)$ is not convex in general, so that
the results of \cite{MWW} do not apply directly.)
 A natural
conjecture is that the gradients $\nabla E_i(x)$ listed
in the order
$ i = 1,\dots,n$
form a totally positive matrix for all $x \in \mathcal{D}_n.$
More generally, for fixed $r \in \N$  take
$\Phi_k(x):= F_{k,r}(x), \ k \in \Lambda := \{r,r+1,\dots, nr\} $
and consider the $n\times n$
matrix $M(x)$ of any $n$ gradients $\nabla F_{k_1,r}(x), \dots, \nabla F_{k_n,r}(x)$
in the order $k_1 \leq \dots \leq k_n .$ Is $M(x)$
totally positive for all $x \in \mathcal{D}_n$?  (Note that  the
$3\times 3$ matrix (\ref{tp3}) of Lemma 13 is essentially
the special case $n=3,\ k_1 = r,\ k_2 = k,\ k_3 = 3r.$
The proof of Lemma 13 may generalize by
considering the matrix of gradients of  the generating functions,
$\nabla f_r(x,t_1),\dots, \nabla f_r(x,t_n)$.)\\

\noindent {\Large {\bf 7.
 The relations $\prec_L  $ and $ \prec_F  $ when $n\leq 3$.} }

When $n=1$, it is clear that if $x,y \in  \R_{+}^n$ then each one of
the relations $x\prec_L y $ and $x \prec_F  y$ is equivalent
to the usual order $x \leq y$. When $n=2$, it can be shown that
 each of $x\prec_L y $ and $x \prec_F  y$ is
 equivalent to the two inequalities $x_1x_2 \leq y_1y_2,
\ x_1+x_2 \leq y_1+y_2$ ; we leave this as an exercise.
When $n\geq 3$, we suspect that there is no simple, finite set
of inequalities equivalent to either $x\prec_L y $ or $x \prec_F  y$.
(Similar remarks have been made in \cite[p. 44, Open Problem 6.5]{N}
 regarding the ``trumping relation" (\S8).)
As seen in the proof of Lemma 11, at each fixed point $x \in
\R_{+}^n$, the positive cone of the gradients $\nabla \Psi_{\lambda}(x),  \lambda > 0$
is generated by just $n$ of these gradients. However, the choice
of these $n$ gradients varies with $x$ in a way which
does not seem to be ``integrable" to yield $n$ simple inequalities.
 In the remainder of this section
 we will discuss the $n=3$ case under the extra condition
 $x_1+x_2+x_3=y_1+y_2+y_3$. This condition effectively brings us back
 down to $2$ dimensions and enables us to find
 a finite set of (two) inequalities characterizing $x\prec_L y $ and $x \prec_F  y$,
 as in the case $n=2$.

Consider the family $\{\Phi_\lambda\}_{\lambda \in \Lambda}$ on $\mathcal{D}_3
= \{x\in \R^3 \ | \ x_1\geq x_2 \geq x_3 \geq 0\}$ given by
$\Lambda := \{1,2,3\}$ with
$ \Phi_1(x) = x_1+x_2+x_3, \ \  \Phi_2(x) = x_2+x_3, \ \  \Phi_3(x) = x_1x_2x_3.$
The gradients are
$$ \nabla \Phi_1(x) = (1,1,1), \ \ \nabla \Phi_2(x) = (0,1,1), \ \
\nabla \Phi_3(x) = (x_2x_3,x_1x_3,x_1x_2). $$
The $3 \times 3$ matrix formed by these gradients,
$$M(x):=\left[
\begin{array}{ccc}
1 & 1 & 1 \\
0 & 1 & 1 \\
x_2x_3 & x_1x_3 & x_1x_2
\end{array}
\right],$$
has $\det M(x) \geq 0$ for all $x \in \mathcal{D}_3,$ but it is not totally
positive since some of its $2\times 2$ subdeterminants are negative
at some points of $\mathcal{D}_3.$ Also, it can be checked that no permutation
of the rows will make $M(x)$ totally positive on $\mathcal{D}_3$.
However, $M(x)$ does have the property that each subdeterminant
has constant sign on $\mathcal{D}_3,$ and $M(x)$ has rank 3 on the interior
of $\mathcal{D}_3.$
Does the conclusion of Conjecture 16 hold ?
We have at least the following  result on the level sets
of $\Phi_1$ on $\mathcal{D}_3.$
\begin{prop}
Suppose $x, y \in \mathcal{D}_3$ satisfy $\Phi_1(x)=\Phi_1(y)=:c,
\ \Phi_2(x) \leq \Phi_2(y), \ \Phi_3(x) \leq \Phi_3(y).$
Then there is a $C^1$ path $\gamma : [0,1] \to \{\Phi_1 = c\}\cap \mathcal{D}_3$
such that {\rm (a)} and {\rm(b)} of Conjecture 16 hold, i.e.
$ \gamma(0) =x, \ \gamma(1) = y,$ and
$ \ 0\leq s \leq t \leq 1 \Rightarrow \Phi_i(\gamma(s)) \leq \Phi_i(\gamma(t)),
i=2,3.$
\end{prop}
\begin{proof} (Outline). It is not difficult to sketch all level curves
$\Phi_2 = c_2 $ and $\Phi_3 = c_3$ on the triangle
$T:=\{\Phi_1 = c\}\cap \mathcal{D}_3,$ and then verify by inspection that
the required $\gamma$ always exists. In fact $\gamma$ can always
be made up of pieces of these level curves or the boundary of $T.$
(Note that ``$C^1$" allows $\gamma$ to have intervals of velocity zero.)
\end{proof}
\begin{corollary} If $x,y \in \mathcal{D}_3$
satisfy $\Phi_1(x)=\Phi_1(y)=:c$ then the three relations (a) $x \prec_L y$,
 (b) $x \prec_F y$, and (c) $\Phi_i(x) \leq \Phi_i(y) , i=2,3$ are
 all equivalent.
\end{corollary}
\begin{proof} (Outline). For simplicity, let $x_1x_2x_3 \neq 0.$
We first show that $(c) \Rightarrow (b).$
Assume $(c).$
Let $\gamma$
be the path given by Proposition 17 and let $z = \gamma(t)$ be
a point on the path.
(In particular $z_1\geq z_2\geq z_3 >0.$) We claim that $\nabla F_{k,r}(z) \in
\{c_1 \nabla \Phi_1(z) + c_2 \nabla \Phi_2(z) +c_3 \nabla \Phi_3(z)
\ | \ c_1,c_2,c_3 \geq 0 \}$  for all $F_{k,r}.$
By Corollary 14 and the proof of Lemma 12 we know that
$\nabla F_{k,r}(z)$ is in the positive cone spanned
by the three vectors $(\frac{1}{z_1},\frac{1}{z_1},\frac{1}{z_1}),
(\frac{1}{z_1},\frac{1}{z_2},\frac{1}{z_2}),
(\frac{1}{z_1},\frac{1}{z_2},\frac{1}{z_3}).$ But each of these
is clearly in the positive cone spanned by $(1,1,1), (0,1,1),
(\frac{1}{z_1},\frac{1}{z_2},\frac{1}{z_3}),$
in other words by $\nabla \Phi_1(z), \nabla \Phi_2(z),
\nabla \Phi_3(z),$ thus proving the claim. Now Lemma 10 implies
that $F_{k,r}(x) \leq F_{k,r}(y),$ that is $x \prec_F y.$
(A similar argument shows that $(c) \Rightarrow (a).$)
Next, it remains to show that $(a) \Rightarrow (c)$,
since we already know that $(b) \Rightarrow (a)$ by Theorem 9.
Assume $(a).$ Condition (b) of Theorem 7 clearly implies $x_1x_2x_3
\leq y_1y_2y_3.$ Since $x_1+x_2+x_3 = y_1+y_2+y_3,$ we may
apply the case $p=\infty$ of (\ref{p2}) to see that $x_1 \geq y_1,$
whence $x_2+x_3 \leq y_2+y_3.$
\end{proof}

We now give a characterization of the inequalities
$F_{k,r}(x) \leq F_{k,r}(y)$ when $r$ is held fixed.
\begin{prop}
Fix an integer $r \geq 1.$ If $x,y \in \R_{+}^3$
satisfy $x_1+x_2+x_3=y_1+y_2+y_3$ then the following
two conditions are equivalent:

{\rm(}i{\rm)} $F_{k,r}(x) \leq F_{k,r}(y), \ k = 1,2,\dots $.

{\rm(}ii{\rm)} $ x_1x_2x_3 \leq y_1y_2y_3 \ $ and
$ -(x_1^{r+1}+x_2^{r+1}+x_3^{r+1})
 \leq  -(y_1^{r+1}+y_2^{r+1}+y_3^{r+1}). $
\end{prop}
\begin{proof} (Outline).
Note that $(ii)$ may be stated as $
\ F_{3r,r}(x) \leq F_{3r,r}(y)\ , \  F_{r+1,r}(x) \leq F_{r+1,r}(y),$
so that clearly  $(i) \Rightarrow (ii).$
For $(ii) \Rightarrow (i)$, consider the triangle $T:=
\{\Phi_1 = constant\}\cap \mathcal{D}_3$ containing $x$ and $y$
as in Proposition 17.
It can be shown that there is a path $\gamma$ in $T$ from $x$ to $y$
along which the functions $ x_1x_2x_3$ and $ -(x_1^{r+1}+x_2^{r+1}+x_3^{r+1})$
are both nondecreasing. Next, we claim that for any $z \in \mathcal{D}_3$,
$\nabla F_{k,r}(z)$ is in the positive
cone spanned by $\nabla (z_1+z_2+z_3), \nabla (-z_1^{r+1}-z_2^{r+1}-z_3^{r+1}),
\nabla (z_1z_2z_3).$ By Cramer's rule, this follows from the
positivity of the determinant
$$
\left|
\begin{array}{ccc}
1 & 1 & 1 \\
-z_1^r & -z_2^r & -z_3^r  \\
  \frac{1}{z_1} & \frac{1}{z_2} & \frac{1}{z_3}
\end{array}
\right| $$
and those obtained by replacing a row by $\nabla F_{k,r}(z)$.
This in turn reduces to the positivity of generalized
Vandermonde determinants by computations similar to those seen in the proof
of Lemma 13. Then $(i)$ follows by Lemma 10.
\end{proof}

\noindent {\Large {\bf 8. The relation $x \prec_L y$ and
``tensor-product-assisted majorization".} }

Let $x,y \in \R_{+}^n$ with $\sum x_i = \sum y_i$.
The ``trumping relation" $x \succ_T y$ \cite{JP}, \cite{N}, \cite{DK}
is defined by the condition
that there exist $d \in \N$ and $0 \neq z \in \R_{+}^d$ (depending
on $x$ and $y$) such that
$x \otimes z \succ y \otimes z$, where $x \otimes z \in \R_{+}^{nd}$ denotes the vector
with entries $(x \otimes z)_{i,j}:= x_iz_j\ ,
1\leq i \leq n; 1\leq j \leq d$, and $\succ$ is the usual majorization relation (see \S1).
(The vector $z$ is called the ``catalyst" which ``assists" the majorization.)
If $p$ is a real number, it is easily seen that $x \succ_T y$
implies $||x||_p \leq ||y||_p$ for $-\infty < p \leq 1$, and
$||x||_p \geq ||y||_p$ for $1 \leq p \leq \infty $ (by factoring out
the $p$-norm of the catalyst). Thus it seems natural to look for
connections between the relations $x \succ_T y$ and $x \prec_L y$,
in view of (\ref{p1}) and (\ref{p2}).
We may first of all remark that
$x \prec_L y, \ ||x||_1 = ||y||_1$ does not imply
$x \succ_T y$, by the example
$x =(15,2,2), \ y=(9,9,1)$ mentioned in \S2.3. Clearly $x \succ_T y$ fails, since
$x \succ_T y$ would imply that $ \min x \leq \min y .$
 Another connection is the following.
It was noted in \S2.1 that, when $\sum x_i = \sum y_i$, the
relation $x \prec_L y$ may be defined by the condition
$$\sum_{i=1}^n \int_0^{x_i} \ \varphi(t) \frac{dt}{t}
\geq \sum_{i=1}^n \int_0^{y_i} \ \varphi(t) \frac{dt}{t}$$
for all nondecreasing convex $\varphi : [0,\infty)\to [0,\infty)$.
This may be re-written as
$$\sum_{i=1}^n \int_0^{1} \ \varphi(x_it) \frac{dt}{t}
\geq \sum_{i=1}^n \int_0^{1} \ \varphi(y_it) \frac{dt}{t} \ .$$
The latter may be interpreted as
$x \otimes z \succ y \otimes z$ where $z$ is the function
$z(t)=t$ on the measure space $[0,1]$ equipped with the
measure $\frac{dt}{t}$, and where by definition $x \otimes z$
is assumed to live on the product measure space $\{1,\dots, n\}\times[0,1]$
equipped with the product of counting measure and $\frac{dt}{t}$.
(Also, for convenience, we are here taking the definition of
$f \succ g$ for integrable functions $f,g \geq 0$ on measure spaces to be:
 $\int f = \int g$ and $\int \varphi\circ f \geq
\int \varphi\circ g$ for all nondecreasing convex $\varphi : [0,\infty)\to [0,\infty)$.)
By another change of variable ($t =e^{-u}$), we may instead use the function
$z(u) = e^{-u}$ on $[0,\infty)$ with Lebesgue measure $du$. In conclusion, we note
that Daftuar \cite[Example 4.3.1]{D} has used a discrete version of such an infinite,
exponential catalyst, namely $z= \frac{1}{1-\alpha}(1,\alpha,\alpha^2,\alpha^3,\dots)$
where $0<\alpha<1$, to give an example with $x \otimes z \succ y \otimes z$
but not $x \succ_T y$ ($ x:=(0.5,0.25,0.25),\ y:= (0.4,0.4,0.2),\ \alpha := 2^{-\frac{1}{8}}$).
 By the present discussion, one could say that this example and
the above example concern similar phenomena.

%%%%%%%%%%%%%%%%%%%%%%%%%%%%%%%%%%%%%%%%%%%%%%%%%%%%%%%%%%%%%%

%%%%%%%%%%%%%%%%%%%%%%%%%%%%%%%%%%%%%%%%%%
\end{document}